\documentclass[11pt]{amsart}

\usepackage{graphicx}
\usepackage{mathptmx}

\usepackage{amscd}
\usepackage{amsthm}
\usepackage{amsxtra}
\usepackage{a4wide}
\usepackage{latexsym}
\usepackage{amssymb}
\usepackage{amsfonts}
\usepackage{amsmath}
\usepackage{amsrefs}
\usepackage{mathrsfs}

\usepackage{upref}
\usepackage{txfonts}

\usepackage{todonotes}

\usepackage[bookmarksnumbered, colorlinks, plainpages]{hyperref}
\hypersetup{colorlinks=true,linkcolor=red, anchorcolor=green, citecolor=cyan, urlcolor=red, filecolor=magenta, pdftoolbar=true}

\allowdisplaybreaks

\theoremstyle{plain}
\newtheorem{thm}{Theorem}[section]
\newtheorem{lem}[thm]{Lemma}

\newtheorem{prop}[thm]{Proposition}
\theoremstyle{definition}
\newtheorem{rem}[thm]{Remark}

\newtheorem{defi}[thm]{Definition}

\newtheorem{conv}[thm]{Convention}

\numberwithin{thm}{section}
\numberwithin{equation}{section}

\def\supp{\operatorname{supp}}
\def\rhs{\operatorname{RHS}}
\def\lhs{\operatorname{LHS}}

\def\esup{\operatornamewithlimits{ess\,sup}}

\def\RHS{\operatorname{RHS}}
\def\LHS{\operatorname{LHS}}

\def\R{\mathbb R}
\def\Z{\mathbb Z}

\def\ap{\approx}

\def\qq{\qquad}
\def\rn{\R^n}

\def\O{\Omega}

\def\la{\lambda}

\def\vp{\varphi}

\def\t{\theta}
\def\I{(0,\infty)}

\def\rw{\rightarrow}

\def\uu{\uparrow\uparrow}
\def\dd{\downarrow\downarrow}
\def\ls{\lesssim}
\def\gs{\gtrsim}

\def\R{\mathbb R}

\def\mp{{\mathfrak M}}
\def\W{{\mathcal W}}

\def\dual{\,^{^{\mathsf{c}}}\!}

\def\Btd {\dual\Bt}
\def\Bxd {\dual\Bx}
\def\Byd {\dual {B(0,y)}}

\def\Bt {{B(0,t)}}

\def\Bx {{B(0,x)}}

\begin{document}

\title{Multidimensional bilinear Hardy inequalities}

\author{N. B\.{I}LG\.{I}\c{C}L\.{I}, R.Ch.~Mustafayev, T.~{\"U}nver}

\address{Nevin Bilgi\c{c}li, Department of Mathematics, Faculty of Science and Arts, Kirikkale 	University, 71450 Yahsihan, Kirikkale, Turkey}
\email{nevinbilgicli@gmail.com}

\address{Rza Mustafayev, Department of Mathematics, Faculty of Science, Karamanoglu Mehmetbey University, Karaman, 70100, Turkey}
\email{rzamustafayev@gmail.com}

\address{Tu\u{g}\c{c}e {\"U}nver, Department of Mathematics, Faculty of Science and Arts, Kirikkale University, 71450 Yahsihan, Kirikkale, Turkey}
\email{tugceunver@gmail.com}



\begin{abstract}
Our goal in this paper is to find a characterization of $n$-dimensional bilinear Hardy inequalities
\begin{align*}
\bigg\| \,\int_{B(0,\cdot)} f \cdot \int_{B(0,\cdot)} g \,\bigg\|_{q,u,\I} & \leq C \, \|f\|_{p_1,v_1,\rn} \, \|g\|_{p_2,v_2,\rn}, \quad f,\,g \in \mp^+ (\rn), \notag \\
\intertext{and} \bigg\| \,\int_{\dual B(0,\cdot)} f \cdot \int_{\dual B(0,\cdot)} g \,\bigg\|_{q,u,\I} &\leq C \, \|f\|_{p_1,v_1,\rn} \, \|g\|_{p_2,v_2,\rn}, \quad f,\,g \in \mp^+ (\rn), 
\end{align*}
when $0 < q \le \infty$, $1 \le p_1,\,p_2 \le \infty$ and $u$ and $v_1,\,v_2$ are weight functions on $(0,\infty)$ and $\rn$, respectively.

Since the solution of the first inequality can be obtained from the characterization of the second one by usual change of variables we concentrate our attention on characterization of the latter. The characterization of this inequality is easily obtained for the range of parameters when $p_1 \le q$ using the characterizations of multidimensional weighted Hardy-type inequalites while in the case when $q < p_1$ the problem is reduced to the solution of multidimensional weighted iterated Hardy-type inequality.

To achieve the goal, we characterize the validity of multidimensional weighted iterated Hardy-type inequality
\begin{equation*}
\left\|\left\|\int_{\dual B(0,s)}h(z)dz\right\|_{p,u,(0,t)}\right\|_{q,\mu,\I}\leq
c \|h\|_{\t,v,\I},~ h \in \mathfrak{M}^+(\rn)
\end{equation*}
where $0 < p,\,q < +\infty$, $1\le\t \le\infty$, $u\in \W\I$, $v\in\W(\rn)$ and $\mu$ is a non-negative Borel measure on $\I$. We are able to obtain the characterization under the additional condition that the measure $\mu$ is non-degenerate with respect to $U^{{q} / {p}}$. 

\end{abstract}

\maketitle


\section{Introduction}\label{in}

The aim of this paper is to study the boundedness of $n$-dimensional bilinear Hardy operators $H_2^n: L^{p_1}(w_1) \times L^{p_2}(w_2) \rw L^q (u)$ and $(H_2^n)^*: L^{p_1}(w_1) \times L^{p_2}(w_2) \rw L^q (u)$, defined for all $f_1,\,f_2 \in \mp^+(\rn)$ by
\begin{align*}
H_2^n (f_1,f_2) (t) : & = \int_{B(0,t)} f_1 (x)\,dx \cdot \int_{B(0,t)} f_2 (x)\,dx, \quad t > 0, \\
\intertext{and} \big(H_2^n\big)^* (f_1,f_2) (t) : & = \int_{\dual B(0,t)} f_1 (x)\,dx \cdot \int_{\dual B(0,t)} f_2 (x)\,dx, \quad t > 0,
\end{align*}
that is, to investigate the validity of $n$-dimensional bilinear Hardy inequalities
\begin{align}
\bigg\| \,\int_{B(0,\cdot)} f \cdot \int_{B(0,\cdot)} g \,\bigg\|_{q,u,\I} & \leq C \, \|f\|_{p_1,v_1,\rn} \, \|g\|_{p_2,v_2,\rn}, \quad f,\,g \in \mp^+ (\rn), \label{eq.main0} \\
\intertext{and} \bigg\| \,\int_{\dual B(0,\cdot)} f \cdot \int_{\dual B(0,\cdot)} g \,\bigg\|_{q,u,\I} &\leq C \, \|f\|_{p_1,v_1,\rn} \, \|g\|_{p_2,v_2,\rn}, \quad f,\,g \in \mp^+ (\rn). \label{eq.main}
\end{align}

The motivation of the investigation of $n$-dimensional $m$-linear Hardy ineqalities can be explained, for instance, by the paper \cite{LOPTT}, 
where a weight theory has been developed for a new multi(sub)linear maximal function 
$$
{\mathcal M}(f_1,\cdots, f_m)(x) : = \sup_{Q \ni x} \prod_{i=1}^{m} \frac{1}{|Q|} \int_Q |f_i (y_i)|\,dy_i, \quad x \in \rn,
$$
where the supremum is taken over all cubes in $\rn$ containing $x$ with sides parallel to the coordinate axes, introduced in order to control the multilinear Calder\'{o}n-Zygmund operators. Recall that, this operator is strictly smaller that the $m$-fold product of $M$, that is, the operator $\prod_{i=1}^{m} Mf_j$, where $M$ is the Hardy-Littlewood maximal operator.
Drawing paralells between linear and $m$-linear theories, in our opinion, it will be useful to know a characterization of weight functions for which 
$n$-dimensional $m$-linear Hardy operator
$$
H_m^n (f_1,\cdots, f_m)(t) : = \int_{B(0,t)} f_1 (x)\,dx \cdots \int_{B(0,t)} f_m (x)\,dx, \quad t > 0
$$
is bounded from $L^{p_1}(w_1) \times \cdots \times L^{p_m}(w_m)$ into $L^p(u)$, that is, the inequality
$$
\|H_m^n (f_1,\cdots, f_m)\|_{L^p(u)} \le C \|f_1\|_{L^{p_1}(w_1)} \cdots \|f_m\|_{L^{p_m}(w_m)}
$$
holds.

In one-dimensional case, the bilinear Hardy operator $H_2 \equiv H_2^1$, acting on ${\mathfrak M}^+(0,\infty) \times {\mathfrak M}^+(0,\infty)$, is defined by
\begin{align*}
H_2 (f,g)(x) & = \int_0^x f(t)\,dt \cdot \int_0^x g(t)\,dt. 
\end{align*}
As far as we know, the boundedness of $H_2: {\mathfrak M}^+(0,\infty) \times {\mathfrak M}^+(0,\infty) \rw L^q (u)$, that is, the bilinear Hardy inequality
\begin{equation}\label{eq.4}
\bigg( \int_0^{\infty} \bigg( \int_0^x f \cdot \int_0^x g \bigg)^q u(x)\,dx
\bigg)^{1 / q} \leq C \, \bigg(\int_0^{\infty} f^{p_1}  v_1 \bigg)^{1 / p_1} \, \bigg(\int_0^{\infty} g^{p_2}  v_2 \bigg)^{1 / p_2}, \quad f,\,g \in \mp^+ (0,\infty)
\end{equation}
has not been considered previously in the literature, apart from the following papers: The papers \cite{cwiker} and \cite{graf.tor.2001}  work with general bilinear operators and characterize their boundedness, in the case $1 / q \ge 1 / p_1 + 1 / p_2$, by means of a Schur-type criterion. The boundedness of $H_2: L^{p_1}(v_1) \times L^{p_2}(v_2)\rightarrow L^q (u)$ was characterized recently in \cite{agu.or.ra.2012} via the discretization method, and in \cite{Krep} using the iteration method. The range of exponents in both papers was $1 < p_1,\,p_2,\,q < \infty$. 

As in 1-dimensional case (cf. \cite{Krep}), the characterization of $n$-dimensional bilinear Hardy inequalities can be easily obtained using the characterizations of multidimensional weighted Hardy-type inequalites, when $p_1 \le q$ (see, Theorems \ref{thm.main.0}, \ref{thm.main.00} and \ref{thm.main.000}). In the most difficult case when $q < p_1$, interchanging the suprema and applying the multidimensional weighted Hardy-type inequalities, by integrating by parts, we get that inequality \eqref{eq.main} is equivalent to the inequality
$$
\bigg( \int_0^{\infty}  \bigg( \int_0^x \bigg( \int_{\Btd} g \bigg)^q u(t) \,dt \bigg)^{r_1 / q} \, d \, \bigg[ - \big\| v_1^{- 1 / p_1} \big\|_{p_1^{\prime},\Bxd}^{r_1} \bigg] \bigg)^{1 / r_1} \le C \|g\|_{p_2,v_2,\rn}, \quad g \in \mp^+ (\rn)
$$
with $1 / r_1 = 1 / q - 1 / p_1$ (see, Theorem \ref{main.thm}).

In this paper we characterize the validity of the multidimensional weighted iterated Hardy-type inequality
\begin{equation}\label{mainn}
\left\|\left\|\int_{\dual B(0,s)}h(z)dz\right\|_{p,u,(0,t)}\right\|_{q,\mu,[0,\infty)}\leq
c \|h\|_{\t,v,\I},~ h \in \mathfrak{M}^+(\rn),
\end{equation}
where $0 < p,\,q <\infty$, $1\le\t \le\infty$, $u\in \W\I$, $v\in\W(\rn)$ and $\mu$ is a non-negative Borel measure on $\I$ (see, Theorem \ref{Thm.4.1}). We are able to obtain the characterization under the additional condition that the measure $\mu$ is non-degenerate with respect to $U^{{q} / {p}}$, that is, conditions  \eqref{eq.nondeg} are satisfied. 

In 1-dimensional case there exist different solutions of iterated Hardy-type inequalities
\begin{equation}\label{mainn0}
\left\| \left\|\int_t^{{\infty}} h(\tau) \, d\tau \right\|_{p,u,(0,\cdot)}\right\|_{q,w,(0,\infty)}\leq C
\,\|h\|_{\theta,v,(0,\infty)}, ~ h \in \mathfrak{M}^+(0,\infty),
\end{equation}
where $0 < p,\,q \leq \infty$, $1 \le \theta \le \infty$ and $u,w,v\in \W\I$.  

Note that inequality \eqref{mainn0} have been considered in the case $p=1$ in \cite{gop2009} (see also \cite{g1}), where the result was
presented without proof, in the case $p=\infty$ in \cite{gop} and in the case $\t=1$ in \cite{gjop} and \cite{ss}, where the special
type of weight function $v$ was considered. Recall that the inequality has been completely characterized	in \cite{gmp} and \cite{gmp2013} in the case $0<p<\infty$, $0<q\leq \infty$, $1 \le \theta \le \infty$ by using discretization and anti-discretization methods. Another approach to get the characterization of inequalities \eqref{mainn0} was presented in \cite{prok.step.2013}.	But these characterizations involve auxiliary functions, which make conditions more complicated. The characterization of the inequality can be reduced to the characterization of the weighted Hardy	inequality on the cones of non-increasing functions (see, \cite{gog.mus.2017_1,gog.mus.2017_2}). Different approach to solve iterated Hardy-type inequalities has been given in \cite{mus.2017}. In order to characterize inequality \eqref{mainn} we will use the technique from 	\cite{gmp} and \cite{gmp2013}.

It should be noted that none of the above would ever have existed if
it wasn't for the (now classical) well-known characterizations of
weights for which the Hardy inequality holds. This subject, which
is, incidentally, exactly one hundred years old, is absolutely
indispensable   in this part of mathematics  (cf. \cite{ok,kp}). In our proof below multidimensional analogues of such
results will be heavily used from \cite{ChristGraf,DrabHeinKuf,mu.emb}.	
	
The paper is organized as follows. We start with some notations and preliminaries in Section~\ref{pre}.  The discretization and anti-discretization methods for solution of inequalities \eqref{mainn} are given in Sections~\ref{mr} and \ref{Antidiscretization}, respectively. Finally, the solutions of multidimensional bilinear Hardy inequalities are presented in Section \ref{main}.

\section{Notations and Preliminaries}\label{pre}

Throughout the paper, we always denote by  $c$ or $C$ a positive
constant, which is independent of the main parameters but it may
vary from line to line. However a constant with subscript such as
$c_1$ does not change in different occurrences. By $a\lesssim b$,
($b\gtrsim a$) we mean that $a\leq \la b$, where $\la >0$ depends
only on inessential parameters. If $a\lesssim b$ and $b\lesssim
a$, we write $a\approx b$ and say that $a$ and $b$ are
equivalent. Throughout the paper we use the abbreviation $\LHS
(*)$ ($\RHS(*)$) for the left (right) hand side of the relation
$(*)$. By $\chi_Q$ we denote the characteristic function of a set
$Q$. Unless a special remark is made, the differential element
$dx$ is omitted when the integrals under consideration are the
Lebesgue integrals.

For $x\in \rn$ and $r>0$, let $B(x,r):=\{y\in \rn: |x-y|< r\}$ be the open ball centered at $x$ of radius $r$ and $\dual B(x,r):= \rn \backslash B(x,r)$. We define $S[a,b):=\{x\in \rn: a\le |x|< b\}=\dual B(0,a) \backslash \dual B(0,b)$, where $0 \le a < b < \infty$.

Let $\Omega$ be any measurable subset of $\rn$, $n\geq 1$. Let $\mu$ be a non-negative measure on $\Omega$. By
$\mp(\Omega,\mu)$ we denote the set of all $\mu$-measurable functions on $\Omega$. The symbol $\mp^+ (\Omega,\mu)$ stands for the
collection of all $f\in\mp (\Omega,\mu)$ which are non-negative on
$\Omega$. The family of all weight functions (also called just weights) on
$\Omega$, that is, locally integrable with respect to measure $\mu$ non-negative functions on
$\Omega$, is given by $\W(\Omega,\mu)$. If the measure $\mu$ is the Lebesgue measure on $I$, then we omit the symbol $\mu$ in the notation. 

For $p\in (0,\infty]$ and $w\in \mp^+(\Omega,\mu)$, we define the functional
$\|\cdot\|_{p,w,\Omega,\mu}$ on $\mp (\Omega,\mu)$ by
\begin{equation*}
\|f\|_{p,w,\Omega,\mu} : = \left\{\begin{array}{cl}
\left(\int_{\Omega} |f(x)|^p w(x)\,d\mu(x) \right)^{1/p} & \qq\mbox{if}\qq p<\infty, \\
\esup_{\Omega} |f(x)|w(x) & \qq\mbox{if}\qq p=\infty.
\end{array}
\right.
\end{equation*}

If, in addition, $w\in \W(\Omega,\mu)$, then the weighted Lebesgue space
$L^p(w,\Omega,\mu)$ is given by
\begin{equation*}
L^p(w,\Omega,\mu) = \{f\in \mp (\Omega,\mu):\,\, \|f\|_{p,w,\Omega,\mu} <
\infty\}
\end{equation*}
and it is equipped with the quasi-norm $\|\cdot\|_{p,w,\Omega,\mu}$.

When $w\equiv 1$ on $\Omega$, we write simply $L^p(\Omega)$ and
$\|\cdot\|_{p,\Omega}$ instead of $L^p(w,\Omega)$ and
$\|\cdot\|_{p,w,\Omega}$, respectively.

We denote for $u,v \in \W \I$ and $1 \le \theta \le \infty$ by
$$
U(t) : =\int_0^t u(s)ds,\qq V_{\t}(t) :=
\left\{\begin{array}{cl}
\big\|v^{-{1} / {\t}}\big\|_{\t',\Btd}, & ~~ \mbox{when}\,\, \theta <\infty, \vspace{0.1cm}\\
\big\|v^{-1}\big\|_{1,\Btd}, & ~~ \mbox{when}\,\,
\t=\infty,
\end{array}
\right. \qq  t\in\I,
$$
and assume that $U(t) > 0$, $t\in\I$.

\begin{conv}\label{Notat.and.prelim.conv.1.1}
{\rm (i)} Throughout the paper we put $0 \cdot \infty = 0$, $\infty / \infty =
0$ and $0/0 = 0$.

{\rm (ii)} If $\theta \in [1,+\infty]$, we define $\theta'$ by $1 / \theta + 1 / \theta' = 1$.

{\rm (iii)} If $I = (a,b) \subseteq \R$ and $g$ is a monotone
function on $I$, then by $g(a)$ and $g(b)$ we mean the limits
$\lim_{x\rw a+}g(x)$ and $\lim_{x\rw b-}g(x)$, respectively.
\end{conv}

Let us now recall some definitions and basic facts concerning
discretization and anti-discretization which can be found in
\cite{gp1}, \cite{gp2} and \cite{gjop}.
\begin{defi}\label{def.2.1}
Let $\{a_k\}$ be a sequence of positive real numbers. We say that
$\{a_k\}$ is geometrically increasing or geometrically decreasing
and write $a_k\uu$ or $a_k\dd$ when
$$
\inf_{k\in\Z}\frac{a_{k+1}}{a_k}>1 ~~\mbox{or}
~~\sup_{k\in\Z}\frac{a_{k+1}}{a_k}<1,
$$
respectively.
\end{defi}

\begin{defi}\label{def.2.2}
Let $b$ be a continuous strictly increasing function on $[0,\infty)$ such that $b(0)=0$ and $\lim\limits_{t\rightarrow\infty} b(t)=\infty$. Then
we say that $b$ is \emph{admissible}.
\end{defi}

\begin{defi}\label{defi.2.5}
Let $b$ be an admissible function. A function $g$ is called \emph{$b$-quasiconcave} if $g$ is equivalent to an increasing function on $(0,\infty)$ and ${g} / b$ is equivalent to a decreasing function on $(0,\infty)$.
\end{defi}

\begin{defi}\label{defi.2.6}
A $b$-quasiconcave function $g$ is called \emph{non-degenerate} if
$$
\lim_{t\rightarrow 0+} g(t) = \lim_{t\rightarrow\infty} \frac{1}{g(t)} = \lim_{t\rightarrow\infty} \frac{g(t)}{b(t)} = \lim_{t\rightarrow 0+} \frac{b(t)}{g(t)}=0.
$$
The family of non-degenerate $b$-quasiconcave functions will be denoted by $\O_b$. 	
\end{defi}

\begin{defi}\label{def.2.3}
Assume that $b$ is admissible and $g \in \O_b$. We say that $\{x_k\}_{k\in\Z}$ is a \emph{discretizing sequence} for $g$ with respect to $b$ if

(i) $x_0=1$ and $b(x_k)\uu$;

(ii) $g(x_k)\uu$ and $\frac{g(x_k)}{b(x_k)}\dd$;

(iii) there is a decomposition $\Z=\Z_1\cup\Z_2$ such that
$\Z_1\cap\Z_2=\emptyset$ and for every $t\in [x_k,x_{k+1}]$
$$
g(x_k)\thickapprox g(t) ~~\mbox{if} ~~ k\in\Z_1,
$$
$$
\frac{g(x_k)}{b(x_k)}\thickapprox \frac{g(t)}{b(t)} ~~\mbox{if}~~ k\in\Z_2.
$$
\end{defi}

Note that if $g\in\O_b$, then there always exists a discretizing sequence for $g$ with respect to $b$ (see, for instance, \cite[Lemma 2.7]{gp1}).

Finally, if $q\in (0,+\infty]$ and $\{w_k\}=\{w_k\}_{k\in \Z}$ is
a sequence of positive numbers, we denote by $\ell^q(\{w_k\},\Z)$
the following  discrete analogue of a weighted Lebesgue space: if
$ 0<q<+\infty$, then
\begin{align*}
\ell^q(\{w_k\},\Z) & = \left\{ \{a_k\}_{k\in\Z} :\,\,
\|a_k\|_{\ell^q(\{w_k\},\Z)}
:=\left(\sum_{k\in \Z}|a_kw_k|^q\right)^{1 / q}<+\infty \right\} \\
\intertext{and} \ell^\infty(\{w_k\},\Z) & = \left\{ \{a_k\}_{k\in\Z}:\,\,
\|a_k\|_{\ell^\infty(\{w_k\},\Z)}:=\sup_{k\in\Z}|a_kw_k|<+\infty
\right\}.
\end{align*}
If $w_k=1$ for all $k\in\Z$, we write simply $\ell^q(\Z)$ instead
of $\ell^q(\{w_k\},\Z)$.

We quote some known results (see, for instance, \cite[Lemma 3.1 and 3.2]{gp1}). 
\begin{lem}\label{lem.2.3}
Let $q\in (0,+\infty ]$. If $\{ \tau_k\} _{k\in\Z}$ is a
geometrically decreasing sequence, then
\begin{equation*}
\left\| \tau _k \sum_{m \le k} a_m\right\| _{\ell^q(\Z)} \approx \|
\tau _k a_k\| _{\ell^q(\Z)}
\end{equation*}
and
\begin{equation*}
\left\| \tau _k \sup _{m\leq k}a_m \right\| _{\ell^q(\Z)} \approx
\|\tau _ka_k\| _{\ell^q(\Z)}
\end{equation*}
for all non-negative sequences $\{a_k\}_{k\in\Z}$.

Let $\{ \sigma_k\} _{k\in\Z}$ be a geometrically increasing
sequence. Then
\begin{equation*}
\left\| \sigma _k \sum_{m \ge k}a_m\right\| _{\ell^q(\Z)} \approx \|
\sigma _ka_k\| _{\ell^q(\Z)}
\end{equation*}
and
\begin{equation*}
\left\| \sigma _k \sup _{m\geq k}a_m \right\| _{\ell^q(\Z)} \approx
\|\sigma _ka_k\| _{\ell^q(\Z)}
\end{equation*}
for all non-negative sequences $\{ a_k\} _{k\in\Z}$.
\end{lem}

Given two (quasi-)Banach spaces $X$ and $Y$, we write $X
\hookrightarrow Y$ if $X \subset Y$ and if the natural embedding
of $X$ in $Y$ is continuous.

The following statement is discrete version of the classical
Landau resonance theorem. Proof can be found, for example, in
\cite{gp1}.
\begin{prop}\label{prop.2.1}{\rm(\cite[Proposition 4.1]{gp1})}
	Let $0 < \theta,\,q \le +\infty$ and let $\{v_k\}_{k\in\Z}$ and $\{w_k\}_{k\in\Z}$ be
	two sequences of positive numbers. Assume that
	\begin{equation}\label{eq31-4651}
	\ell^{\theta} (\{v_k\},\Z) \hookrightarrow \ell^q (\{w_k\},\Z).
	\end{equation}
	Then
	\begin{equation*}\label{eq31-46519009}
	\big\|\big\{w_k v_k^{-1}\big\}\big\|_{\ell^\rho(\Z)} \le C,
	\end{equation*}
	where $1 / \rho : = ( 1 / q - 1 / \theta)_+$ \footnote{For any $a\in\R$ denote
		by $a_+ = a$ when $a>0$ and $a_+ = 0$ when $a \le 0$.} and $C$ stands for the norm of
	embedding \eqref{eq31-4651}.
\end{prop}

We shall use the following inequality, which is a simple consequence
of the discrete H\"{o}lder inequality:
\begin{equation}\label{discrete.Hold.}
\|\{ a_k b_k \}\|_{\ell^q (\Z)} \le \|\{ a_k \}\|_{\ell^{\rho}
	(\Z)} \|\{ b_k \}\|_{\ell^{\theta} (\Z)}.
\end{equation}


\section{Discretization of Inequality \eqref{mainn}}\label{mr}

In this section we discretize the inequality 
\begin{align}
\left(\int_{[0,\infty)}\left(\frac{1}{U(t)}\int_0^t\left(\int_{\dual B(0,y)} h(z)dz\right)^p u(y)dy\right)^{{q} / {p}} \,d\mu(t) \right)^{{1} / {q}}\leq c \|h\|_{\t,v,\rn}. \label{eq.4.1}
\end{align}
At first we do the following remarks.

\begin{rem}\label{cor.2.0}
	Recall that, if $F$ is a non-negative non-increasing function on $\I$, then
	\begin{equation}\label{Fubini.1}
	\esup_{t \in (0,\infty)} F(t)G(t) = \esup_{t \in (0,\infty)} F(t)
	\esup_{\tau \in (0,t)} G(\tau);	
	\end{equation}
	likewise, when $F$ is a non-negative non-decreasing function on $\I$, then
	\begin{equation}\label{Fubini.2}	
	\esup_{t \in (0,\infty)} F(t)G(t) = \esup_{t \in (0,\infty)} F(t)
	\esup_{\tau \in (t,\infty)} G(\tau)
	\end{equation}
	(see, for instance, \cite[p. 85]{gp2}).
	
	Given a non-negative non-decreasing function $b$ on $\I$ denote by
	$$
	{\mathcal B}(x,t) : = \frac{b(x)}{b(x) + b(t)} \qquad (x > 0,\, t > 0).
	$$
	Observe that
	$$
	{\mathcal B}(x,t) \approx \min \bigg\{ 1,\frac{b(x)}{b(t)} \bigg\}.
	$$	
	It is easy to see that ${\mathcal B}(x,t)$ is $b$-quasiconcave function of $x$ for any fixed $t > 0$.

	It have been shown in \cite[p. 85]{gp2} that the relations
	\begin{align}
	\esup_{t \in (0,\infty)}{\mathcal B}(x,t) g(t) & \approx \esup_{t \in (0,\infty)} g(t) \min \bigg\{ 1,\frac{b(x)}{b(t)} \bigg\} \notag \\
	& = \esup_{t \in (0,x)} b(t) \esup_{\tau \in (t,\infty)} \frac{g(\tau)}{b(\tau)} \notag \\
	& = b(x) \esup_{t \in (x,\infty)} \frac{1}{b(t)} \esup_{\tau \in (0,t)} g(\tau)   \label{Fubini.3}
	\end{align}
	holds for any $g\in {\mathfrak M}^+ (0,\infty)$. Consequently,  $\esup_{t \in (0,\infty)}{\mathcal B}(x,t) g(t)$ is $b$-quasiconcave function. 
\end{rem}

\begin{rem}
Let $0 < p,\,q < \infty$. Suppose that $U$ is admissible on $\I$. Assume that $\mu$ is a non-negative Borel measure on $[0,\infty)$ and $\vp$ is the fundamental function of $\mu$ with respect to $U^{{q} / {p}}$, that is,
\begin{equation}\label{eq.4.3}
\vp(x): =\int_{[0,\infty)}{\mathcal U}(x,y)^{{q} / {p}} d\mu (y)
\qquad\mbox{for all} \qquad x\in \I,
\end{equation}
where
$$
{\mathcal U}(x,t): = \frac{U(x)}{U(t)+U(x)}.
$$
Assume that the measure $\mu$ is non-degenerate with respect to
$U^{{q} / {p}}$:
\begin{equation}\label{eq.nondeg}
\int_{[0,\infty)}\frac{d\mu(t)}{U(t)^{{q} / {p}}+U(x)^{{q} / {p}}}<\infty,~x\in\I~ \mbox{and}
~\int_{[0,1]}\frac{d\mu(t)}{U(t)^{{q} / {p}}}=\int_{[1,\infty)}d\mu(t)=\infty.
\end{equation}
Then $\vp\in\O_{U^{{q} / {p}}}$, and therefore
there exists a discretizing sequence for $\vp$ with respect to
$U^{{q} / {p}}$. Let $\{x_k\}$ be one such sequence. Then
$\vp(x_k)\uu$ and $\vp(x_k)U^{-{q} / {p}}\dd$. Furthermore, there
is a decomposition $\Z=\Z_1\cup\Z_2$, $\Z_1\cap\Z_2=\emptyset$ such
that for every $k\in\Z_1$ and $t\in[x_k,x_{k+1}]$, $\vp (x_k)
\ap\vp(t)$ and for every $k\in\Z_2$ and $t\in[x_k,x_{k+1}]$,
$\vp(x_k){U(x_k)}^{-{q} / {p}}\ap\vp(t){U(t)}^{-{q} / {p}}$ (see \cite[Remark 2.10]{gp1}).
\end{rem}

\begin{lem}\label{Lem3.0}
Let $0 < p,\,q < \infty$ and let $u,\,w\in {\mathcal W}\I$. Assume that $u$ is such that $U$ is admissible. Suppose that non-negative Borel measure $\mu$ on $[0,\infty)$ is non-degenerate with respect
to $U^{{q} / {p}}$. Let $\{x_k\}$ be any discretizing sequence for the fundamental function $\vp$ of $\mu$ with respect to $U^{{q} / {p}}$. Then	
\begin{align*}
\lhs \eqref{eq.4.1}  \, \ap 
\left\|\left\{\left\|\int_{S[y,x_k)}h(z)dz\right\|_{p,u,I_k}\frac{\vp(x_k)^{{1} / {q}}}{U(x_k)^{{1} / {p}}}\right\}\right\|_{\ell^q(\Z)} + \left\|\left\{\vp(x_k)^{{1} / {q}}\int_{S[x_k,x_{k+1})} h(z)dz \right\}\right\|_{\ell^q(\Z)}.
\end{align*}
\end{lem}
\begin{proof}
Applying \cite[Corollary 2.13]{gp1} to the $U$-quasiconcave function
$$
f(t)=\int_0^t\left(\int_{\dual B(0,y)}h(z)dz\right)^p u(y)dy,
$$
we get that
\begin{equation*}\label{eq134-6813=0968}
\lhs \eqref{eq.4.1} \ap \left\|\left\{\left\|\int_{\dual B(0,y)} h(z)dz\right\|_{p,u,(0,x_k)}\frac{\vp(x_k)^{{1} / {q}}}{U(x_k)^{{1} / {p}}}\right\}\right\|_{\ell^q(\Z)}.
\end{equation*}	
Using Lemma \ref{lem.2.3},
\begin{align*}
\lhs \eqref{eq.4.1}  \ap &
\left\|\left\{\left\|\int_{S[y,\infty)} h(z)dz\right\|_{p,u,I_k} \frac{\vp(x_k)^{{1} / {q}}}{U(x_k)^{{1} / {p}}}\right\}\right\|_{\ell^q(\Z)}
\\
\ap &
\left\|\left\{\left\|\int_{S[y,{x_k})} h(z)dz + \int_{S[x_k,\infty)} h(z)dz\right\|_{p,u,I_k} \frac{\vp(x_k)^{{1} / {q}}}{U(x_k)^{{1} / {p}}}\right\}\right\|_{\ell^q(\Z)}
\\
\ap &
\left\|\left\{\left\|\int_{S[y,x_k)}h(z)dz\right\|_{p,u,I_k}\frac{\vp(x_k)^{{1} / {q}}}{U(x_k)^{{1} / {p}}}\right\}\right\|_{\ell^q(\Z)} + \left\|\left\{\left\|\int_{S[x_k,\infty)} h(z)dz\right\|_{p,u,I_k}\frac{\vp(x_k)^{{1} / {q}}}{U(x_k)^{{1} / {p}}}\right\}\right\|_{\ell^q(\Z)},
\end{align*}
where $I_k : = [x_{k-1},x_k)$, $k \in \Z$. Since $\|1\|_{p,u,I_k}^p	\ap U(x_k)$,
we obtain that
\begin{align*}
\lhs \eqref{eq.4.1}  \ap \left\|\left\{\left\|\int_{S[y,x_k)} h(z)dz\right\|_{p,u,I_k}\frac{\vp(x_k)^{{1} / {q}}}{U(x_k)^{{1} / {p}}}\right\}\right\|_{\ell^q(\Z)} + \left\|\left\{\vp(x_k)^{{1} / {q}}\int_{S[x_k,\infty)} h(z)dz \right\}\right\|_{\ell^q(\Z)}.
\end{align*}
By using Lemma \ref{lem.2.3} on the second term, we arrive at
\begin{align*}
\lhs \eqref{eq.4.1}  \, \ap
\left\|\left\{\left\|\int_{S[y,x_k)}h(z)dz\right\|_{p,u,I_k}\frac{\vp(x_k)^{{1} / {q}}}{U(x_k)^{{1} / {p}}}\right\}\right\|_{\ell^q(\Z)} + \left\|\left\{\vp(x_k)^{{1} / {q}}\int_{S[x_k,x_{k+1})} h(z)dz \right\}\right\|_{\ell^q(\Z)}.
\end{align*}	
\end{proof}

\begin{lem}\label{Lem3.1}
Let $0 < p,q < \infty$, $1\le \t \le \infty$, $1 / \rho = (1 / q - 1 / \t)_+$, and let $u,\,w\in {\mathcal W}\I$ and $v\in {\mathcal W}(\rn)$ be such that $U$ is admissible and $V_{\theta}(t) < \infty$, $t \in \I$ with $\lim_{t\rightarrow \infty} V_{\theta} (t) = 0$. Suppose that non-negative Borel measure $\mu$ on $[0,\infty)$ is non-degenerate with respect to $U^{{q} / {p}}$. Let $\{x_k\}$ be any discretizing sequence for the fundamental function $\vp$ of the measure $\mu$ with respect to $U^{{q} / {p}}$. Then inequality \eqref{eq.4.1} holds  for every $h \in \mp^+(\rn)$ if and only if
\begin{equation}\label{A}
A : = \left\| \left\{
\frac{\vp(x_k)^{{1} / {q}}} {U(x_k)^{{1} / {p}}} B(x_{k-1},x_k) \right\}
\right\|_{\ell^{\rho}(\Z)} + \left\| \left\{\vp(x_k)^{{1} / {q}}C(x_k,x_{k+1})\right\} \right\|_{\ell^{\rho}(\Z)} < \infty,
\end{equation}
where  
\begin{align}
B(x_{k-1},x_k) : & = \sup_{h\in \mp^+(S[x_{k-1},x_k))} \left\|\int_{S[t,x_k)}
h(z)dz\right\|_{p,u,[x_{k-1},x_{k})} \, / \, \left\|h\right\|_{\t,v,S[x_{k-1},x_k)}, \label{Bk} \\
\intertext{and}
C(x_k,x_{k+1}) : & = \sup_{h\in \mp^+(S[x_k,x_{k+1}))} \|h\|_{1,S[x_k,x_{k+1})} \, / \, \|h\|_{\t,v,S[x_k,x_{k+1})}. \label{defi.C} 
\end{align}
Moreover, the best constant in inequality \eqref{eq.4.1} satisfies $ c\ap A$.
\end{lem}

\begin{proof}
{\textbf{Sufficiency}.} 
In view of \eqref{Bk} and inequality \eqref{discrete.Hold.}, we have that
\begin{align}\label{I}
\left\|\left\{\left\|\int_{S[y,x_k)}h(z)dz\right\|_{p,u,I_k}\frac{\vp(x_k)^{{1} / {q}}}{U(x_k)^{{1} / {p}}}\right\}\right\|_{\ell^q(\Z)}  & \le \left\|\left\{B(x_{k-1},x_k)
\frac{\vp(x_k)^{{1} / {q}}}{U(x_k)^{{1} / {p}}}\|h\|_{\t,v,S[x_{k-1},x_k)}\right\}\right\|_{\ell^q(\Z)} \notag\\
& \le \left\|\left\{B(x_{k-1},x_k)
\frac{\vp(x_k)^{{1} / {q}}}{U(x_k)^{{1} / {p}}}\right\}\right\|_{\ell^{\rho}(\Z)}
\left\|\{\|h\|_{\t,v,S[x_{k-1},x_k)}\}\right\|_{\ell^{\t}(\Z)} \nonumber\\
& = \left\|\left\{B(x_{k-1},x_k) \frac{\vp(x_k)^{{1} / {q}}}{U(x_k)^{{1} / {p}}}\right\}\right\|_{\ell^{\rho}(\Z)} \|h\|_{\t,v,\rn}.
\end{align}
By \eqref{defi.C} and \eqref{discrete.Hold.}, we get that
\begin{align}
\left\|\left\{\vp(x_k)^{{1} / {q}}\int_{S[x_k,x_{k+1})} h(z)dz
\right\}\right\|_{\ell^q(\Z)} & \leq \left\|\left\{\vp(x_k)^{{1} / {q}} C(x_k,x_{k+1}) \|h\|_{\t,v,S[x_k,x_{k+1})}\right\}\right\|_{\ell^q(\Z)} \nonumber\\
& \lesssim
\left\|\left\{\vp(x_k)^{{1} / {q}}C(x_k,x_{k+1})\right\}
\right\|_{\ell^{\rho}(\Z)}\|h\|_{\t,v,\rn}. \label{II}
\end{align}
By Lemma \ref{Lem3.0}, using \eqref{I} and \eqref{II}, we obtain 
\begin{align*}
\lhs \eqref{eq.4.1} & \lesssim \left( \left\| \left\{\frac{\vp(x_k)^{{1} / {q}}} {U(x_k)^{{1} / {p}}} B(x_{k-1},x_k) \right\}
\right\|_{\ell^{\rho}(\Z)} + \left\| \left\{\vp(x_k)^{{1} / {q}}C(x_k,x_{k+1})\right\}
\right\|_{\ell^{\rho}(\Z)} \right) \|h\|_{\t,v,\rn} = A \, \|h\|_{\t,v,\rn}.
\end{align*}
Consequently, \eqref{eq.4.1} holds provided that $A<\infty$ and $c \le A$.

{\textbf{Necessity}.} Assume that inequality
\eqref{eq.4.1} holds with $c < \infty$. By \eqref{Bk}, there are $h_k \in {\mathfrak M}^+(\rn)$, $k\in \Z$, such that $\supp h_k \subset S[x_{k-1},x_k)$,
\begin{equation}\label{hknorm}
\|h_k\|_{\t,v,S[x_{k-1},x_k)} = 1 \quad \mbox{and} \quad 
\left\|\int_{S[y,x_k)} h_k(z)dz\right\|_{p,u,I_k} \ge \frac{1}{2} B(x_{k-1},x_k) \qquad \mbox{for
all}\qquad k\in\Z.
\end{equation}
Define
\begin{equation}\label{g}
h = \sum_{m\in \Z} a_m h_m,
\end{equation}
where $\{a_k\}_{k\in\Z}$ is any sequence of positive numbers. Then, by Lemma \ref{Lem3.0}, we have that
\begin{align}
\lhs \eqref{eq.4.1}  \gs
\left\|\left\{\left\|\int_{S[y,x_k)} \sum_{m\in \Z} a_m h_m \right\|_{p,u,I_k} \frac{\vp(x_k)^{{1} / {q}}}{U(x_k)^{{1} / {p}}} \right\}\right\|_{\ell^q(\Z)}
\gs \left\|\left\{a_k B(x_{k-1},x_k)\frac{\vp(x_k)^{{1} / {q}}}{U(x_k)^{{1} / {p}}}\right\}\right\|_{\ell^q(\Z)}.
\label{nec1}
\end{align}
Moreover,
\begin{align}\label{nec2}
\rhs \eqref{eq.4.1}  = c\left\|\sum_{m\in \Z} a_m
h_m\right\|_{\t,v,\rn} = c\left\|\left\{a_k\right\}\right\|_{\ell^{\t}(\Z)}.
\end{align}
By \eqref{eq.4.1}, \eqref{nec1} and \eqref{nec2}, we
obtain that
\begin{equation}
\left\|\left\{a_k
B(x_{k-1},x_k)\frac{\vp(x_k)^{{1} / {q}}}{U(x_k)^{{1} / {p}}}\right\}\right\|_{\ell^q(\Z)}
\ls c \left\|\left\{a_k\right\}\right\|_{\ell^{\t}(\Z)}.
\end{equation}
Then, by Proposition \ref{prop.2.1} we arrive at
\begin{equation}\label{eq25-609820586y}
\left\| \left\{
\frac{\vp(x_k)^{{1} / {q}}}{U(x_k)^{{1} / {p}}}B(x_{k-1},x_k)\right\} \right\|_{\ell^{\rho}(\Z)} \lesssim c.
\end{equation}

On the other hand, by \eqref{defi.C}, there are $\psi_k \in {\mathfrak M}^+(\rn)$, $k\in \Z$,
such that $\supp \psi_k \subset S[x_k,x_{k+1})$,
\begin{equation}\label{eq13487561134551}
\|\psi_k\|_{\t,v,S[x_k,x_{k+1})} = 1
\quad \mbox{and} \quad \|\psi_k\|_{1,S[x_k,x_{k+1})} \ge  \frac{1}{2} C(x_k,x_{k+1}) \qquad
\mbox{for all}\qquad k\in\Z.
\end{equation}
Define
\begin{equation}\label{f}
h = \sum_{m\in \Z} b_m \psi_m,
\end{equation}
where $\{b_k\}_{k\in\Z}$ is any sequence of positive numbers. Then, by Lemma \ref{Lem3.0}, we have that
\begin{align*}
\lhs \eqref{eq.4.1} \gs
\left\|\left\{\vp(x_k)^{{1} / {q}}\int_{S[x_k,x_{k+1})} \sum_{m\in \Z} b_m \psi_m \right\}\right\|_{\ell^q(\Z)}
\gs \left\|\left\{b_k\vp(x_k)^{{1} / {q}}C(x_k,x_{k+1})\right\}\right\|_{\ell^q(\Z)}.
\end{align*}
We also have,
\begin{align*}
\rhs \eqref{eq.4.1}  = c \left\|\sum_{m\in \Z} b_m
\psi_m\right\|_{\t,v,\rn} = c
\left\|\left\{b_k\right\}\right\|_{\ell^{\t}(\Z)}.
\end{align*}
Consequently
$$
\left\|\left\{b_k
\vp(x_k)^{{1} / {q}}C(x_k,x_{k+1})\right\}\right\|_{\ell^q(\Z)}
\ls c \left\|\left\{b_k\right\}\right\|_{\ell^{\t}(\Z)}.
$$
Then, applying Proposition \ref{prop.2.1}, we get that
\begin{equation}\label{eq205689728967}
\left\|\left\{\vp(x_k)^{{1} / {q}}C(x_k,x_{k+1})\right\}\right\|_{\ell^{\rho}(\Z)} \lesssim c.
\end{equation}
Combining \eqref{eq25-609820586y} and \eqref{eq205689728967}, we arrive at $A \ls c$.
\end{proof}

\begin{rem}\label{remconv}
Let $1 \le \theta \le \infty$. Note that
\begin{equation}\label{Ck}
C(x_{k},x_{k+1}) = \left\{\begin{array}{ll}
\big\|v^{-{1} / {\t}}\big\|_{\t',S[x_k,x_{k+1})}, & ~~ \mbox{when}\,\, \t <\infty, \vspace{0.1cm}\\
\big\|v^{-1}\big\|_{1,S[x_k,x_{k+1})}, & ~~ \mbox{when}\,\,
\t=\infty,
\end{array}
\right. \qquad k \in \Z.
\end{equation}	
If $\theta < \infty$, in view of Lemma \ref{lem.2.3}, it is evident
that
\begin{align*}
\left\|\left\{\vp(x_k)^{{1} / {q}}C(x_k,x_{k+1})\right\}
\right\|_{\ell^{\rho}(\Z)} & = \left\| \left\{
\vp(x_k)^{{1} / {q}} \big \|v^{-{1} / {\t}} \big\|_{\t',S[x_k,x_{k+1})}\right\}
\right\|_{\ell^{\rho}(\Z)} \\
& \ap \left\| \left\{
\vp(x_k)^{{1} / {q}} \big \|v^{-{1} / {\t}} \big\|_{\t',S[x_k,\infty)}\right\}
\right\|_{\ell^{\rho}(\Z)}.
\end{align*}
Monotonicity of $\big\|v^{-{1} / {\t}} \big\|_{\t',S[t,\infty)}$ implies that
\begin{align*}
\left\| \left\{ \vp(x_k)^{{1} / {q}}
\big \|v^{-{1} / {\t}} \big\|_{\t',S[x_k,\infty)}\right\} \right\|_{\ell^{\rho}(\Z)}  \ge
\left\| \left\{ \vp(x_k)^{{1} / {q}} \right\}
\right\|_{\ell^{\rho}(\Z)} \lim_{t\rightarrow
\infty}\big\|v^{-{1} / {\t}} \big\|_{\t',S[t,\infty)}.
\end{align*}
Since $\left\{ \vp(x_k)^{{1} / {q}} \right\}$ is geometrically
increasing, we obtain that
$$
\left\| \left\{ \vp(x_k)^{{1} / {q}}
\big\|v^{-{1} / {\t}} \big\|_{\t',S[x_k,\infty)}\right\} \right\|_{\ell^{\rho}(\Z)}  \ge
\vp(\infty)^{{1} / {q}} \lim_{t\rightarrow \infty}\big\|v^{-{1} / {\t}} \big\|_{\t',S[t,\infty)}.
$$
This inequality shows that $\lim_{t\rightarrow \infty} \|v^{-{1} / {\t}} \|_{\t',S[t,\infty)}$ must be equal to $0$, because $\vp(\infty)$  is
always equal to $\infty$ by our assumptions on the function $\vp$. 

Similarly, $\lim_{t\rightarrow \infty} \|v^{-1} \|_{1,S[t,\infty)}$ must be equal to $0$, when $\theta = \infty$.

Therefore, throughout the paper we consider weight functions $v$ such that $\lim_{t\rightarrow \infty} V_{\t}(t) = 0$. 

Note also that the condition $V_{\theta} (t) < \infty$, $t \in \I$ implies $\lim_{t\rightarrow \infty} V_{\t}(t) = 0$, when $1 < \theta \le \infty$.
\end{rem}


\section{Anti-discretization of conditions}\label{Antidiscretization}

In this section we anti-discretize the conditions obtained in Lemma
\ref{Lem3.1}.

\begin{lem}\label{Lem4.2}
Let $0 < p,\, q < \infty$, $1\le \t \le \infty$, $1 / \rho =\left(1/q-1/\t\right)_+$ and let $u,\,w\in {\mathcal W}\I$ and $v\in {\mathcal W}(\rn)$ be such that $U$ is admissible and $V_{\theta}(t) < \infty$, $t \in \I$ with $\lim_{t\rightarrow \infty} V_{\theta} (t) = 0$. Suppose that non-negative Borel measure $\mu$ on $[0,\infty)$ is non-degenerate with respect to $U^{{q} / {p}}$. Let $\{x_k\}$ be any discretizing sequence for the fundamental function $\vp$ of the measure $\mu$ with respect to $U^{{q} / {p}}$. 

{\rm (a)} If $\t \le p$, then $A\ap A^*$, where
\begin{align*}
A^*:=\left\| \left\{ \vp(x_k)^{{1} / {q}} \bigg(\sup_{t\in(0,\infty)} U(t,x_k)^{{1} / {p}} V_{\theta} (t) \bigg)\right\}\right\|_{\ell^{\rho}(\Z)}.
\end{align*}

{\rm (b)} If $p < \theta$ and $1 /r = 1 / p - 1 / \t $, then	
$$
A  \ap B^*,
$$
where
\begin{align*}
B^* : =  \left\| \left\{ \vp(x_k)^{{1} / {q}} \left(\int_{[0,\infty)}
{\mathcal U}(t,x_k)^{{r} / {p}} d \left( - V_{\theta}(t-)^r\right) \right)^{{1} / {r}}\right\}
\right\|_{\ell^{\rho}(\Z)}.
\end{align*}
Here
$$
V_{\theta}(t-) : = \lim_{\tau \rw t-} V_{\theta}(\tau).
$$
\end{lem}
\begin{proof}
(a) By \cite[Theorem 2.2, (a) and (f)]{mu.emb}, from Lemma \ref{Lem3.1}, we have that
\begin{align*}
A \ap & \left\| \left\{ \frac{\vp(x_k)^{{1} / {q}}}{U(x_k)^{{1} / {p}}} \sup_{t \in I_k} \bigg(\int_{x_{k-1}}^t u(s)ds\bigg)^{{1} / {p}} \big\|v^{-{1} / {\t}}\big\|_{\t',S[t,x_k)} \right\}\right\|_{\ell^{\rho}(\Z)} \\
& + \left\| \left\{ \vp(x_k)^{{1} / {q}} \big\|v^{-{1} / {\t}}\big\|_{\t',S[x_k,x_{k+1})} \right\}\right\|_{\ell^{\rho}(\Z)}.
\end{align*}
By Lemma \ref{lem.2.3}, we get that
\begin{align*}
A \lesssim & \left\| \left\{ \frac{\vp(x_k)^{{1} / {q}}}{U(x_k)^{{1} / {p}}} \sup_{t \in I_k} U(t)^{{1} / {p}} \big\|v^{-{1} / {\t}}\big\|_{\t',S[t,\infty)} \right\}\right\|_{\ell^{\rho}(\Z)} + \left\| \left\{ \vp(x_k)^{{1} / {q}} \big\|v^{-{1} / {\t}}\big\|_{\t',S[x_k,x_{k+1})} \right\}\right\|_{\ell^{\rho}(\Z)} \\
\approx & \left\| \left\{ \frac{\vp(x_k)^{{1} / {q}}}{U(x_k)^{{1} / {p}}} \sup_{t \in (0,x_k)} U(t)^{{1} / {p}} \big\|v^{-{1} / {\t}}\big\|_{\t',S[t,\infty)} \right\}\right\|_{\ell^{\rho}(\Z)} + \left\| \left\{ \vp(x_k)^{{1} / {q}} \big\|v^{-{1} / {\t}}\big\|_{\t',S[x_k,\infty)} \right\}\right\|_{\ell^{\rho}(\Z)} \\
= & \left\| \left\{ \frac{\vp(x_k)^{{1} / {q}}}{U(x_k)^{{1} / {p}}} \sup_{t \in (0,x_k)} U(t)^{{1} / {p}} \big\|v^{-{1} / {\t}}\big\|_{\t',S[t,\infty)} \right\}\right\|_{\ell^{\rho}(\Z)} + \left\| \left\{ \vp(x_k)^{{1} / {q}} \bigg(\sup_{t \in [x_k,\infty)} \big\|v^{-{1} / {\t}}\big\|_{\t',S[t,\infty)}\bigg) \right\}\right\|_{\ell^{\rho}(\Z)} \\
\ap & \left\| \left\{ \vp(x_k)^{{1} / {q}} \bigg(\sup_{t \in (0,\infty)} U(t,x_k)^{{1} / {p}} \big\|v^{-{1} / {\t}}\big\|_{\t',S[t,\infty)}\bigg) \right\}\right\|_{\ell^{\rho}(\Z)} = A^*.
\end{align*}

We now prove the reverse estimate. We have that
\begin{align*}
A^* \approx &  \left\| \left\{ \frac{\vp(x_k)^{{1} / {q}}}{U(x_k)^{{1} / {p}}} \sup_{t \in I_k} U(t)^{{1} / {p}} \big\|v^{- {1} / {\t}}\big\|_{\t',S[t,\infty)} \right\}\right\|_{\ell^{\rho}(\Z)} + \left\| \left\{ \vp(x_k)^{{1} / {q}}\big\|v^{-{1} / {\t}}\big\|_{\t',S[x_k,\infty)} \right\}\right\|_{\ell^{\rho}(\Z)} \\
\ls &  \left\|\left\{\frac{\vp(x_k)^{{1} / {q}}}{U(x_k)^{{1} / {p}}} \left(\sup_{t \in I_k}  \bigg(\int_{x_{k-1}}^{t} u(s)ds\bigg)^{{1} / {p}} \big\|v^{-{1} / {\t}}\big\|_{\t',S[t,\infty)}\right) \right\}\right\|_{\ell^{\rho}(\Z)} \\
& + \left\|\left\{\frac{\vp(x_k)^{{1} / {q}}}{U(x_k)^{{1} / {p}}} U(x_{k-1})^{{1} / {p}} \big\|v^{-{1} / {\t}}\big\|_{\t',S[x_{k-1},\infty)} \right\}\right\|_{\ell^{\rho}(\Z)} + \left\| \left\{ \vp(x_k)^{{1} / {q}}\big\|v^{-{1} / {\t}}\big\|_{\t',S[x_k,\infty)} \right\}\right\|_{\ell^{\rho}(\Z)} \\
\ls & \left\|\left\{\frac{\vp(x_k)^{{1} / {q}}}{U(x_k)^{{1} / {p}}} \left(\sup_{t \in I_k}  \bigg(\int_{x_{k-1}}^{t} u(s)ds\bigg)^{{1} / {p}} \big\|v^{-{1} / {\t}}\big\|_{\t',S[t,\infty)}\right) \right\}\right\|_{\ell^{\rho}(\Z)} \\
& + \left\|\left\{\vp(x_{k-1})^{{1} / {q}}  \big\|v^{-{1} / {\t}}\big\|_{\t',S[x_{k-1},\infty)} \right\}\right\|_{\ell^{\rho}(\Z)} + \left\|\left\{\vp(x_k)^{{1} / {q}}  \big\|v^{-{1} / {\t}}\big\|_{\t',S[x_k,\infty)} \right\}\right\|_{\ell^{\rho}(\Z)} \\
\approx & \left\|\left\{\frac{\vp(x_k)^{{1} / {q}}}{U(x_k)^{{1} / {p}}} \left(\sup_{t \in I_k}  \bigg(\int_{x_{k-1}}^{t} u(s)ds\bigg)^{{1} / {p}} \big\|v^{-{1} / {\t}}\big\|_{\t',S[t,\infty)}\right) \right\}\right\|_{\ell^{\rho}(\Z)} \\
& + \left\|\left\{\vp(x_k)^{{1} / {q}}  \big\|v^{-{1} / {\t}}\big\|_{\t',S[x_k,\infty)} \right\}\right\|_{\ell^{\rho}(\Z)}\\
\ls & \left\|\left\{\frac{\vp(x_k)^{{1} / {q}}}{U(x_k)^{{1} / {p}}} \left(\sup_{t \in I_k}  \bigg(\int_{x_{k-1}}^{t} u(s)ds\bigg)^{{1} / {p}} \big\|v^{-{1} / {\t}}\big\|_{\t',S[t,x_k)}\right) \right\}\right\|_{\ell^{\rho}(\Z)}\\
& + \left\|\left\{\vp(x_k)^{{1} / {q}}  \big\|v^{-{1} / {\t}}\big\|_{\t',S[x_k,\infty)} \right\}\right\|_{\ell^{\rho}(\Z)}\\
\ap & \left\|\left\{\frac{\vp(x_k)^{{1} / {q}}}{U(x_k)^{{1} / {p}}} \left(\sup_{t \in I_k}  \bigg(\int_{x_{k-1}}^{t} u(s)ds\bigg)^{{1} / {p}} \big\|v^{- {1} / {\t}}\big\|_{\t',S[t,x_k)}\right) \right\}\right\|_{\ell^{\rho}(\Z)}\\
& + \left\|\left\{\vp(x_k)^{{1} / {q}}  \big\|v^{-{1} / {\t}}\big\|_{\t',S[x_k,x_{k+1})} \right\}\right\|_{\ell^{\rho}(\Z)} = A.
\end{align*}

(b) Assume that $\theta < \infty$. By \cite[Theorem 2.2, (b) and (g)]{mu.emb}, and \eqref{Ck}, from Lemma \ref{Lem3.1}, we have that
\begin{align*}
A  \ap & \left\| \left\{
\frac{\vp(x_k)^{{1} / {q}}}{U(x_k)^{{1} / {p}}} \left(\int_{x_{k-1}}^{x_k} \left( \int_{x_{k-1}}^t u(s)\,ds\right)^{{r} / {\t}}
u(t) \big\|v^{-{1} / {\t}}\big\|_{\t',S[t,x_k)}^r \,dt
\right)^{{1} / {r}}\right\}\right\|_{\ell^{\rho}(\Z)} \notag\\
&  + \left\|\left\{ \vp(x_k)^{\frac{1}{q}}
\big\|v^{-{1} / {\t}}\big\|_{\t',S[x_k,x_{k+1})}\right\}\right\|_{\ell^{\rho}(\Z)}.
\end{align*}
Since
$$
\left(\int_{x_{k-1}}^{x_k} \bigg(\int_{x_{k-1}}^t u(s) ds\bigg)^{{r} / {\t}} u(t) dt \right)^{{1} / {r}} \approx U(x_k)^{1 / p},
$$
it is easy to see that
\begin{align*}
A \ap & \left\| \left\{
\frac{\vp(x_k)^{{1} / {q}}} {U(x_k)^{{1} / {p}}} \left(\int_{x_{k-1}}^{x_k} \left( \int_{x_{k-1}}^t u(s)\,ds\right)^{{r} / {\t}}
u(t)\big\|v^{-{1} / {\t}}\big\|_{\t',S[t,x_k)}^r \,dt\right)^{{1} / {r}}\right\}
\right\|_{\ell^{\rho}(\Z)}  \\
& + \left\| \left\{ \frac{\vp(x_k)^{{1} / {q}}} {U(x_k)^{{1} / {p}}} \left(\int_{x_{k-1}}^{x_k} \bigg(\int_{x_{k-1}}^t u(s) ds\bigg)^{{r} / {\t}} u(t) dt \right)^{{1} / {r}} \big\|v^{-{1} / {\t}}\big\|_{\t',S[x_k,x_{k+1})} \right\} \right\|_{\ell^{\rho}(\Z)}  \\
\lesssim & \left\| \left\{
\frac{\vp(x_k)^{{1} / {q}}}{U(x_k)^{{1} / {p}}} \left(\int_{x_{k-1}}^{x_k} \left( \int_{x_{k-1}}^t u(s)\,ds\right)^{{r} / {\t}} u(t) \big\|v^{-{1} / {\t}}\big\|_{\t',S[t,\infty)}^r \,dt \right)^{{1} / {r}}\right\} \right\|_{\ell^{\rho}(\Z)} \\
\le & \left\| \left\{ \frac{\vp(x_k)^{{1} / {q}}}{U(x_k)^{{1} / {p}}} \left(\int_{x_{k-1}}^{x_k} U(t)^{{r} / {\t}} u(t) \big\|v^{-{1} / {\t}}\big\|_{\t',S[t,\infty)}^r \,dt \right)^{{1} / {r}}\right\} \right\|_{\ell^{\rho}(\Z)} \\
\ap & \left\| \left\{
\frac{\vp(x_k)^{{1} / {q}}}{U(x_k)^{{1} / {p}}} \left(\int_{x_{k-1}}^{x_k}  \big\|v^{-{1} / {\t}}\big\|_{\t',S[t,\infty)}^r \,d\bigg(U(t)^{{r} / {p}}\bigg) \right)^{{1} / {r}}\right\}
\right\|_{\ell^{\rho}(\Z)}.
\end{align*}
Integrating by parts, we arrive at 
\begin{align}
A \ls & \left\| \left\{
\frac{\vp(x_k)^{{1} / {q}}}{U(x_k)^{{1} / {p}}} \left(\int_{[x_{k-1},x_k)} U(t)^{{r} / {p}} \,d\bigg(-\big\|v^{-{1} / {\t}}\big\|_{\t',S[t-,\infty)}^r\bigg) \right)^{{1} / {r}}\right\}
\right\|_{\ell^{\rho}(\Z)}\notag \\
& + \left\| \left\{
\vp(x_k)^{{1} / {q}} \big\|v^{-{1} / {\t}}\big\|_{\t',S[{x_k}-,\infty)} \right\}
\right\|_{\ell^{\rho}(\Z)} \label{1}
\end{align}
By Lemma~\ref{lem.2.3}, in view of Remark~\ref{remconv}, we have that
\begin{align}
\left\| \left\{
\vp(x_k)^{\frac{1}{q}} \big\|v^{-{1} / {\t}}\big\|_{\t',S[{x_k}-,\infty)} \right\}
\right\|_{\ell^{\rho}(\Z)} = \notag & \\
& \hspace{-3cm} = \left\| \left\{
\vp(x_k)^{{1} / {q}} \left(\big\|v^{-{1} / {\t}}\big\|_{\t',S[{x_k}-,\infty)}^r-\lim_{t\rw \infty}\big\|v^{-{1} / {\t}}\big\|_{\t',S[{x_k}-,\infty)}^r\right)^{1 / {r}}  \right\}
\right\|_{\ell^{\rho}(\Z)} \notag \\
&\hspace{-3cm} \ap \left\| \left\{
\vp(x_k)^{{1} / {q}} \left(\sum_{i=k}^\infty \bigg(\big\|v^{-{1} / {\t}}\big\|_{\t',S[{x_i}-,\infty)}^r - \big\|v^{-{1} / {\t}}\big\|_{\t',S[{x_{i+1}}-,\infty)}^r\bigg)\right)^{1 / {r}}  \right\}\right\|_{\ell^{\rho}(\Z)} \notag\\
&\hspace{-3cm} \ap \left\| \left\{\vp(x_k)^{{1} / {q}} \left( \big\|v^{-{1} / {\t}}\big\|_{\t',S[{x_k}-,\infty)}^r - \big\|v^{-{1} / {\t}}\big\|_{\t',S[{x_{k+1}}-,\infty)}^r\bigg)\right)^{1 / {r}}  \right\}\right\|_{\ell^{\rho}(\Z)} \notag \\
& \hspace{-3cm} \ap \left\| \left\{\vp(x_k)^{{1} / {q}} \left( \int_{[x_k,x_{k+1})} d\bigg(-\big\|v^{-{1} / {\t}}\big\|_{\t',S[t-,\infty)}^r\bigg)\right)^{1 / {r}}  \right\}
\right\|_{\ell^{\rho}(\Z)}. \label{ibparts}
\end{align}
Using \eqref{ibparts} in \eqref{1} and applying Lemma~\ref{lem.2.3}, we arrive at
\begin{align*}
A\ls & \left\| \left\{
\frac{\vp(x_k)^{{1} / {q}}}{U(x_k)^{{1} / {p}}} \left(\int_{[0,x_k)} U(t)^{{r} / {p}} \,d\bigg(-\big\|v^{-{1} / {\t}}\big\|_{\t',S[t-,\infty)}^r\bigg) \right)^{{1} / {r}}\right\}
\right\|_{\ell^{\rho}(\Z)}\notag \\
& \qq + \left\| \left\{
\vp(x_k)^{{1} / {q}}\left( \int_{[x_k,\infty)} d\bigg(-\big\|v^{-{1} / {\t}}\big\|_{\t',S[t-,\infty)}^r\bigg)\right)^{1 / {r}} \right\}
\right\|_{\ell^{\rho}(\Z)}\\
\ap & \left\| \left\{
\vp(x_k)^{{1} / {q}} \left(\int_{[0,\infty)} U(t,x_k)^{{r} / {p}} \,d\bigg(-\big\|v^{-{1} / {\t}}\big\|_{\t',S[t-,\infty)}^r\bigg) \right)^{{1} / {r}}\right\}
\right\|_{\ell^{\rho}(\Z)} = B^*.
\end{align*}
Consequently, $A\ls B^*$.

Conversely, by Lemma~\ref{lem.2.3}, in view of Remark \ref{remconv}, we have that
\begin{align*}
B^*\ap &\left\| \left\{ \frac{\vp(x_k)^{{1} /{q}}}{U(x_k)^{{1} / {p}}} \left(\int_{[x_{k-1},x_k)} U(t)^{{r} / {p}} d\bigg(-\big\|v^{-{1} / {\t}}\big\|_{\t',S[t-,\infty)}^r\bigg)\right)^{{1} / {r}} \right\} \right\|_{\ell^{\rho}(\Z)}\\
& + \left\| \left\{ \vp(x_k)^{{1} / {q}} \big\|v^{-{1} / {\t}}\big\|_{\t',S[x_k-,\infty)}\right\} \right\|_{\ell^{\rho}(\Z)} \\
\ls & \left\| \left\{ \frac{\vp(x_k)^{{1} / {q}}}{U(x_k)^{{1} / {p}}} \left(\int_{[x_{k-1},x_k)}  \bigg(\int_{x_{k-1}}^t u(s)ds\bigg)^{{r} / {p}} d\bigg(-\big\|v^{-{1} / {\t}}\big\|_{\t',S[t-,\infty)}^r\bigg)\right)^{{1} / {r}} \right\} \right\|_{\ell^{\rho}(\Z)}\\
& + \left\| \left\{ \frac{\vp(x_k)^{{1} /{q}}}{U(x_k)^{{1} / {p}}} U(x_{k-1})^{{1} / {p}} \big\|v^{-{1} / {\t}}\big\|_{\t',S[x_{k-1}-,\infty)}\right\} \right\|_{\ell^{\rho}(\Z)} + \left\| \left\{ \vp(x_k)^{{1} / {q}} \big\|v^{-{1} / {\t}}\big\|_{\t',S[x_k-,\infty)}\right\} \right\|_{\ell^{\rho}(\Z)}\\
\ls & \left\| \left\{ \frac{\vp(x_k)^{{1} / {q}}}{U(x_k)^{{1} / {p}}} \left(\int_{[x_{k-1},x_k)}  \bigg(\int_{x_{k-1}}^t u(s)ds\bigg)^{{r} / {p}} d\bigg(-\big\|v^{-{1} / {\t}}\big\|_{\t',S[t-,\infty)}^r\bigg)\right)^{{1} / {r}} \right\} \right\|_{\ell^{\rho}(\Z)}\\
& + \left\| \left\{ \vp(x_{k-1})^{{1} /{q}} \big\|v^{-{1} / {\t}}\big\|_{\t',S[x_{k-1}-,\infty)}\right\} \right\|_{\ell^{\rho}(\Z)} + \left\| \left\{ \vp(x_k)^{{1} / {q}} \big\|v^{-{1} / {\t}}\big\|_{\t',S[x_k-,\infty)}\right\} \right\|_{\ell^{\rho}(\Z)}\\
\approx & \left\| \left\{ \frac{\vp(x_k)^{{1} / {q}}}{U(x_k)^{{1} / {p}}} \left(\int_{[x_{k-1},x_k)}  \bigg(\int_{x_{k-1}}^t u(s)ds\bigg)^{{r} / {p}} d\bigg(-\big\|v^{-{1} / {\t}}\big\|_{\t',S[t-,\infty)}^r\bigg)\right)^{{1} / {r}} \right\} \right\|_{\ell^{\rho}(\Z)}\\
& + \left\| \left\{ \vp(x_k)^{{1} / {q}} \big\|v^{-{1} / {\t}}\big\|_{\t',S[x_k-,\infty)}\right\} \right\|_{\ell^{\rho}(\Z)}.
\end{align*}
Integrating by parts yields that
\begin{align*}
B^* \ls & \left\| \left\{ \frac{\vp(x_k)^{{1} / {q}}}{U(x_k)^{{1} / {p}}} \left(\int_{[x_{k-1},x_k)} \big\|v^{-{1} / {\t}}\big\|_{\t',S[t,\infty)}^r   d\bigg(\bigg(\int_{x_{k-1}}^t u(s)ds\bigg)^{{r} / {p}}\bigg)\right)^{{1} / {r}} \right\} \right\|_{\ell^{\rho}(\Z)}\\
& + \left\| \left\{ \vp(x_k)^{{1} / {q}} \big\|v^{-{1} / {\t}}\big\|_{\t',S[x_k-,\infty)}\right\} \right\|_{\ell^{\rho}(\Z)}\\
\ap & \left\| \left\{ \frac{\vp(x_k)^{{1} / {q}}}{U(x_k)^{{1} / {p}}} \left(\int_{[x_{k-1},x_k)} \bigg(\int_{x_{k-1}}^t u(s)ds\bigg)^{{r} / {\t}} u(t) \big\|v^{-{1} / {\t}}\big\|_{\t',S[t,\infty)}^r   dt\right)^{{1} / {r}} \right\} \right\|_{\ell^{\rho}(\Z)}\\
& + \left\| \left\{ \vp(x_k)^{{1} / {q}} \big\|v^{-{1} / {\t}}\big\|_{\t',S[x_k-,\infty)}\right\} \right\|_{\ell^{\rho}(\Z)}.    
\end{align*}
Since
\begin{align*}
\left\| \left\{ \vp(x_k)^{{1} / {q}} \big\|v^{-{1} / {\t}}\big\|_{\t',S[x_k-,\infty)}\right\}\right\|_{\ell^{\rho}(\Z)} & \\
&\hspace{-3cm} = \left\| \left\{ \vp(x_{k-1})^{{1} / {q}} \big\|v^{-{1} / {\t}}\big\|_{\t',S[x_{k-1}-,\infty)}\right\}\right\|_{\ell^{\rho}(\Z)}\\
&\hspace{-3cm}  \ap \left\| \left\{ \frac{\vp(x_{k-1})^{{1} / {q}}}{U(x_{k-1})^{{1} / {p}}} \big\|v^{-{1} / {\t}}\big\|_{\t',S[x_{k-1}-,\infty)} \left(\int_{x_{k-2}}^{x_{k-1}} \bigg(\int_{x_{k-2}}^t u(s)ds\bigg)^{{r} / {\t}} u(t) dt\right)^{{1} / {r}} \right\} \right\|_{\ell^{\rho}(\Z)}\\
&\hspace{-3cm}  \le \left\| \left\{ \frac{\vp(x_{k-1})^{{1} / {q}}}{U(x_{k-1})^{{1} / {p}}}  \left(\int_{x_{k-2}}^{x_{k-1}} \bigg(\int_{x_{k-2}}^t u(s)ds\bigg)^{{r} / {\t}} u(t) \big\|v^{-{1} / {\t}}\big\|_{\t',S[t-,\infty)}^r dt\right)^{{1} / {r}} \right\} \right\|_{\ell^{\rho}(\Z)}\\
&\hspace{-3cm}  = \left\| \left\{ \frac{\vp(x_{k})^{{1} / {q}}}{U(x_{k})^{{1} / {p}}}  \left(\int_{x_{k-1}}^{x_{k}} \bigg(\int_{x_{k-1}}^t u(s)ds\bigg)^{{r} / {\t}} u(t) \big\|v^{-{1} / {\t}}\big\|_{\t',S[t-,\infty)}^r    dt\right)^{{1}/{r}} \right\} \right\|_{\ell^{\rho}(\Z)},
\end{align*}
we arrive at
\begin{align*}
B^* \ls & \left\| \left\{ \frac{\vp(x_{k})^{{1} / {q}}}{U(x_{k})^{{1} / {p}}}  \left(\int_{x_{k-1}}^{x_{k}} \bigg(\int_{x_{k-1}}^t u(s)ds\bigg)^{{r} / {\t}} u(t) \big\|v^{-{1} / {\t}}\big\|_{\t',S[t,\infty)}^r dt\right)^{{1} / {r}} \right\} \right\|_{\ell^{\rho}(\Z)}\\
\ap & \left\| \left\{ \frac{\vp(x_{k})^{{1} / {q}}}{U(x_{k})^{{1} / {p}}}  \left(\int_{x_{k-1}}^{x_{k}} \bigg(\int_{x_{k-1}}^t u(s)ds\bigg)^{{r} / {\t}} u(t) \big\|v^{-{1} / {\t}}\big\|_{\t',S[t,x_k)}^r  dt\right)^{{1} / {r}} \right\} \right\|_{\ell^{\rho}(\Z)}\\
& +\left\| \left\{ \vp(x_{k})^{{1} / {q}} \big\|v^{-{1} / {\t}}\big\|_{\t',S[x_k,\infty)}  \right\} \right\|_{\ell^{\rho}(\Z)} \ap A.    
\end{align*}	

Now assume that $\theta = \infty$. In this case the proof can be done in the same line and we leave it to the reader. The only difference is that one should apply \cite[Theorem 2.2, (e)]{mu.emb} and take into account that $C(x_k,x_{k+1}) = \|v^{-1}\|_{1,S[x_k,x_{k+1})}$, $k \in \Z$.
\end{proof}

We now in a position to characterize inequality \eqref{eq.4.1}.
\begin{thm}\label{Thm.4.1}
Let $0 < p,\, q < \infty$, $1\le \t \le \infty$,  $1 / \rho =\left(1/q-1/\t\right)_+$  and let $u,\,w\in {\mathcal W} \I$ and $\in {\mathcal W}(\rn)$ be such that $U$ is admissible and $V_{\theta}(t) < \infty$, $t \in \I$ with $\lim_{t\rightarrow \infty} V_{\theta} (t) = 0$. Suppose that non-negative Borel measure $\mu$ on $[0,\infty)$ is non-degenerate with respect to $U^{{q} / {p}}$. Then the inequality \eqref{eq.4.1} holds for every measurable function on $\rn$ if and only if
\begin{itemize}
\item[(i)] $\t \le \min\{p,q\}$ and
$$
I_1:= \sup_{x\in \I} \bigg(\int_0^\infty{\mathcal U}(x,t)^{{q} / {p}} \,d\mu(t) \bigg)^{{1} / {q}} \sup_{t\in(0,\infty)} {\mathcal U}(t,x)^{{1} / {p}} V_{\theta}(t)<\infty.
$$
Moreover, the best constant in \eqref{eq.4.1} satisfies $c\ap I_1$.

\item[(ii)] $q < \t < p$ and
$$
I_2:=\left(\int_0^\infty \bigg(\int_0^\infty {\mathcal U}(x,t)^{{q} / {p}} \,d\mu (t)\bigg)^{{\rho} / {\t}} \, \bigg( \sup_{t\in(0,\infty)}{\mathcal U}(t,x)^{1 / {p}}  V_{\theta}(t) \bigg)^{\rho} \,d\mu (x) \right)^{{1} / {\rho}}<\infty.
$$
Moreover, the best constant in \eqref{eq.4.1} satisfies $c\ap I_2$.

\item[(iii)] $p < \t \le q$, $r=\t p/ (\t-p)$ and
$$
I_3:=\sup_{x\in \I} \bigg(\int_0^\infty {\mathcal U}(x,t)^{{q} / {p}} \,d\mu(t) \bigg)^{{1} / {q}} \bigg(\int_0^\infty {\mathcal U}(t,x)^{{r} / {p}} \, d \, \big(-V_{\theta}(t-)^r\big)\bigg)^{{1} / {r}}<\infty.
$$
Moreover, the best constant in \eqref{eq.4.1} satisfies $c\ap I_3$.

\item[(iv)] $\max\{p,q\} < \t$, $r=\t p/ (\t-p)$ and
$$
I_4:=\left(\int_0^\infty \bigg(\int_0^\infty {\mathcal U}(x,t)^{{q} / {p}} \,d\mu (t)\bigg)^{{\rho} / {\t}} \bigg(\int_0^\infty {\mathcal U}(t,x)^{{r} / {p}} \, d \, \big(-V_{\theta}(t-)^r\big) \bigg)^{{\rho} / {r}} \,d\mu(x)\right)^{{1} / {\rho}} < \infty.
$$
Moreover, the best constant in \eqref{eq.4.1} satisfies $c\ap I_4$.

\item[(v)] $\t = \infty$ and
$$
I_5:=\left(\int_0^\infty\left( \int_0^\infty {\mathcal U(t,x)}\, d\big(-V_{\infty}(t)^p\big)\right)^{{q} / {p}} \, d\mu (x)\right)^{{1} / {q}} <\infty
$$
Moreover, the best constant in \eqref{eq.4.1} satisfies $c\ap I_5$.

\end{itemize}
\end{thm}

\begin{proof}
\noindent \begin{itemize}
\item[(i)] The proof of the statement follows by Lemma~\ref{Lem3.1}, Lemma~\ref{Lem4.2}, (a), and \cite[Lemma 3.5]{gp1}.

\item[(ii)] The proof of the statement follows by Lemma~\ref{Lem3.1}, Lemma~ \ref{Lem4.2}, (a) and \cite[Theorem 2.11]{gp1}.

\item[(iii)] The proof of the statement follows by Lemma~\ref{Lem3.1}, Lemma~ \ref{Lem4.2}, (b), and \cite[Lemma 3.5]{gp1}.

\item[(iv)] The proof of the statement follows by Lemma~\ref{Lem3.1}, Lemma~\ref{Lem4.2}, (b), and \cite[Theorem 2.11]{gp1}.

\item[(v)] The proof of the statement follows by Lemma~\ref{Lem3.1}, Lemma~\ref{Lem4.2}, (b), and \cite[Theorem 2.11]{gp1}.
\end{itemize}
\end{proof}



\section{Characterization of $n$-dimensional bilinear Hardy inequalities}\label{main}

In this section we give characterization of $n$-dimensional bilinear Hardy inequalities \eqref{eq.main0} and \eqref{eq.main}.
 
The following  note allows us to concentrate our attention only on characterization of \eqref{eq.main}.

	\begin{rem}\label{equiv.rem}
		Note that the inequality
		\begin{equation}\label{11}
		\bigg( \int_0^{\infty} \bigg(\int_{B(0,t)} f \cdot \int_{B(0,t)} g \bigg)^q u(t) dt \bigg)^{{1} / {q}} \leq C \, \bigg( \int_{\rn} f^{p_1} v_1 \bigg)^{{1} / {p_1}} \bigg(\int_{\rn} g^{p_2} v_2 \bigg)^{{1} / {p_2}}
		\end{equation}	
		is equivalent to the inequality
		\begin{equation}\label{2}
		\bigg( \int_0^{\infty} \bigg(\int_{\dual B(0,t)} f \cdot \int_{\Btd} g \bigg)^q \tilde u(t) dt\bigg)^{{1} / {q}} \leq C \bigg(\int_{\rn} f^{p_1} \tilde v_1 \bigg)^{{1} / {p_1}} \bigg(\int_{\rn} g^{p_2} \tilde v_2 \bigg)^{{1} / {p_2}},
		\end{equation}	
		where $\tilde u(t) = u\big(t^{-1}\big) t^{-2}$, $\tilde v_1(x) = v_1\big(|x|^{-2} x\big) |x|^{-2n(1-p_1)}$, $\tilde v_2(x) = v_2\big(|x|^{-2} x\big) |x|^{-2n(1-p_2)}$.
		
		Indeed: Since any $f \in \mp^+ (\rn)$ can be uniquely represented as $f(x) = g \big(|x|^{-2}x \big)|x|^{-2n}$, $g \in \mp^+(\rn)$, then inequality \eqref{11} is equivalent to the following inequality
		\begin{align}
		\bigg( \int_0^{\infty} \bigg(\int_{B(0,t)} f\big(|y|^{-2} y\big) |y|^{-2n} \,dy \cdot \int_{B(0,t)} g\big(|y|^{-2} y\big) |y|^{-2n} \,dy \bigg)^q u(t) dt\bigg)^{{1} / {q}} & \notag \\
		&\hspace{-8cm} \leq C \, \bigg(\int_{\rn} \big( f \big(|y|^{-2} y\big)\big)^{p_1} v_1(y) |y|^{-2np_1} dy \bigg)^{{1} / {p_1}} \bigg(\int_{\rn} \big( g\big(|y|^{-2} y\big) \big)^{p_2} v_2(y) |y|^{-2np_2}\,dy \bigg)^{{1} / {p_2}}. \label{3}
		\end{align}
		Using the substitution $x=|y|^{-2} y$ in multidimensional integrals, we get that \eqref{3} is equivalent to the inequality
		\begin{align*}
		\bigg(\int_0^{\infty}\bigg(\int_{\dual B(0,{1} / {t})} f(x) dx \cdot \int_{\dual B(0,{1} / {t})} g(x) dx \bigg)^q u(t) dt\bigg)^{{1} / {q}} & \\
		&\hspace{-7cm}\leq C \, \bigg(\int_{\rn} f(x)^{p_1} v_1\big(|x|^{-2} x\big) |x|^{-2n(1-p_1)} dx \bigg)^{{1} / {p_1}} \bigg(\int_{\rn} g(x)^{p_2} v_2\big(|x|^{-2} x\big) |x|^{-2n(1-p_2)} \,dx \bigg)^{{1} / {p_2}},
		\end{align*}
		and finally applying $\tau = {1} / {t}$, we see that the latter is equivalent to
		\begin{align*}
		\bigg( \int_0^{\infty}\bigg(\int_{\dual B(0,\tau)} f(x) dx \cdot \int_{\dual B(0,\tau)} g(x) dx \bigg)^q u\big(\tau^{-1}\big)  \tau^{-2}\, d\tau \bigg)^{{1} / {q}} & \\
		&\hspace{-7cm} \leq C \, \bigg(\int_{\rn} f(x)^{p_1} v_1\big(|x|^{-2} x\big) |x|^{-2n(1-p_1)} dx \bigg)^{{1} / {p_1}} \bigg(\int_{\rn} g(x)^{p_2} v_2\big(|x|^{-2} x\big) |x|^{-2n(1-p_2)} \,dx\bigg)^{{1} / {p_2}}.
		\end{align*}
	\end{rem}
	
	Now we present and prove our main results.
	\begin{thm}\label{thm.main.0}
	Let $0 < q < \infty$, $1 \le p_1,\,p_2 \le \infty$, $p_1 \le q$, and let $u\in {\mathcal W}(0,\infty)$, $v_1,\,v_2 \in {\mathcal W}(\rn)$. Then inequality \eqref{eq.main}
	holds for all $f,\,g \in \mp^+ (\rn)$ if and only if:
	
	{\rm (a)} $1 \le p_2 \le q < \infty$, and
	\begin{align*}
	B_1 : = & \sup_{t \in \I} U(t)^{1 / q} \big\| v_1^{- 1 / p_1} \big\|_{p_1^{\prime},\Btd}  \big\| v_2^{- 1 / p_2} \big\|_{p_2^{\prime},\Btd} < \infty.
	\end{align*}
	Moreover, the best constant $C$ in \eqref{eq.main} satisfies $C\approx B_1$.
	
    {\rm (b)} $1 \le p_2  < \infty$, $0 < q < p_2$, $1 / r_2 = 1 / q -  1 / p_2$, and
    \begin{align*}
    B_2 : = &  \sup_{t \in \I} \big\| v_1^{- 1 / p_1} \big\|_{p_1^{\prime},\Btd} \bigg( \int_0^t U(y)^{r_2 / p_2} u(y) \big\| v_2^{- 1 / p_2} \big\|_{p_2^{\prime},\Byd}^{r_2} \,dy \bigg)^{1 / r_2} < \infty.
    \end{align*}
    Moreover, the best constant $C$ in \eqref{eq.main} satisfies $C\approx B_2$.
    
    {\rm (c)} $p_2 = \infty$, and
    \begin{align*}
    B_3 : = &  \sup_{t \in \I} \big\| v_1^{- 1 / p_1} \big\|_{p_1^{\prime},\Btd} \bigg( \int_0^t u(y) \big\| v_2^{- 1} \big\|_{1,\Byd}^q \,dy \bigg)^{1 / q} < \infty.
    \end{align*}
    Moreover, the best constant $C$ in \eqref{eq.main} satisfies $C\approx B_3$.
	\end{thm}
	
	\begin{proof}
	Interchanging the suprema, we obtain that
	\begin{align}
	\sup_{f,g \in \mp^+(\rn)} \frac{\bigg\| \,\int_{\dual B(0,\cdot)} f \cdot \int_{\dual B(0,\cdot)} g \,\bigg\|_{q,u,\I}}{\|f\|_{p_1,v_1,\rn} \, \|g\|_{p_2,v_2,\rn}} & \notag \\
	& \hspace{-5cm} = \sup_{g \in \mp^+ (\rn)} \frac{1}{\|g\|_{p_2,v_2,\rn}} \sup_{f \in \mp^+(\rn)} \frac{\bigg\| \,\int_{\dual B(0,\cdot)} f \cdot \int_{\dual B(0,\cdot)} g \,\bigg\|_{q,u,\I}} {\|f\|_{p_1,v_1,\rn}}. \label{eq.iteration}
	\end{align}
	By \cite[Theorem 2.2, (a) and (f)]{mu.emb},  we get that
	\begin{align*}
	\sup_{f,g \in \mp^+(\rn)} \frac{\bigg\| \,\int_{\dual B(0,\cdot)} f \cdot \int_{\dual B(0,\cdot)} g \,\bigg\|_{q,u,\I}}{\|f\|_{p_1,v_1,\rn} \, \|g\|_{p_2,v_2,\rn}} & \\
	& \hspace{-5cm} = \sup_{g \in \mp^+ (\rn)} \frac{1}{\|g\|_{p_2,v_2,\rn}} \sup_{t \in \I} \bigg( \int_0^t \bigg( \int_{\dual B(0,\tau)} g \bigg)^q u(\tau)\,d\tau \bigg)^{1 / q} \, \big\| v_1^{- 1 / p_1} \big\|_{p_1^{\prime},\Btd}.
	\end{align*}
	
	(a) Let $1 \le p_2 \le q < \infty$. Again, interchanging the suprema, by \cite[Theorem 2.2, (a) and (f)]{mu.emb}, on using \eqref{Fubini.1}, we get that
	\begin{align*}
	\sup_{f,g \in \mp^+(\rn)} \frac{\bigg\| \,\int_{\dual B(0,\cdot)} f \cdot \int_{\dual B(0,\cdot)} g \,\bigg\|_{q,u,\I}}{\|f\|_{p_1,v_1,\rn} \, \|g\|_{p_2,v_2,\rn}} & \\
	& \hspace{-5cm} = \sup_{t \in \I}{\big\| v_1^{- 1 / p_1} \big\|_{p_1^{\prime},\Btd}} \sup_{g \in \mp^+ (\rn)} \frac{\bigg( \int_0^{\infty} \bigg( \int_{\dual B(0,\tau)} g \bigg)^q \chi_{(0,t)}(\tau) u(\tau)\,d\tau \bigg)^{1 / q} }{\|g\|_{p_2,v_2,\rn}} \\
	& \hspace{-5cm} = \sup_{t \in \I} \big\| v_1^{- 1 / p_1} \big\|_{p_1^{\prime},\Btd} \sup_{y \in \I} \bigg( \int_0^y 
	\chi_{(0,t)}(\tau) u(\tau)\,d\tau \bigg)^{1 / q} \big\| v_2^{- 1 / p_2} \big\|_{p_2^{\prime},\Byd} \\
	& \hspace{-5cm} = \sup_{t \in \I} \big\| v_1^{- 1 / p_1} \big\|_{p_1^{\prime},\Btd} \sup_{y \in (0,t)} \bigg( \int_0^y 
	u \bigg)^{1 / q} \big\| v_2^{- 1 / p_2} \big\|_{p_2^{\prime},\Byd} \\
	& \hspace{-5cm} = \sup_{t \in \I} U(t)^{1 / q} \big\| v_1^{- 1 / p_1} \big\|_{p_1^{\prime},\Btd}  \big\| v_2^{- 1 / p_2} \big\|_{p_2^{\prime},\Btd}.
	\end{align*}
	
	(b) Let $1 \le p_2  < \infty$, $0 < q < p_2$  and $1 / r_2 = 1 / q -  1 / p_2$. Interchanging the suprema, by \cite[Theorem 2.2, (b) and (g)]{mu.emb},  we obtain that
	\begin{align*}
	\sup_{f,g \in \mp^+(\rn)} \frac{\bigg\| \,\int_{\dual B(0,\cdot)} f \cdot \int_{\dual B(0,\cdot)} g \,\bigg\|_{q,u,\I}}{\|f\|_{p_1,v_1,\rn} \, \|g\|_{p_2,v_2,\rn}} & \\
    & \hspace{-5cm} = \sup_{t \in \I}{\big\| v_1^{- 1 / p_1} \big\|_{p_1^{\prime},\Btd}} \sup_{g \in \mp^+ (\rn)} \frac{\bigg( \int_0^{\infty} \bigg( \int_{\dual B(0,\tau)} g \bigg)^q \chi_{(0,t)}(\tau) u(\tau)\,d\tau \bigg)^{1 / q} }{\|g\|_{p_2,v_2,\rn}} \\
	& \hspace{-5cm} = \sup_{t \in \I} \big\| v_1^{- 1 / p_1} \big\|_{p_1^{\prime},\Btd} \bigg( \int_0^{\infty} \bigg( \int_0^y \chi_{(0,t)}(\tau) u(\tau)\,d\tau \bigg)^{r_2 / p_2} \chi_{(0,t)}(y) u(y) \big\| v_2^{- 1 / p_2} \big\|_{p_2^{\prime},\Byd}^{r_2} \,dy \bigg)^{1 / r_2} \\
	& \hspace{-5cm} = \sup_{t \in \I} \big\| v_1^{- 1 / p_1} \big\|_{p_1^{\prime},\Btd} \bigg( \int_0^t U(y)^{r_2 / p_2} u(y) \big\| v_2^{- 1 / p_2} \big\|_{p_2^{\prime},\Byd}^{r_2} \,dy \bigg)^{1 / r_2}.
	\end{align*}
	
    {\rm (c)} Let $p_2 = \infty$. Interchanging the suprema, by \cite[Theorem 2.2, (e)]{mu.emb},  we obtain that
    \begin{align*}
    \sup_{f,g \in \mp^+(\rn)} \frac{\bigg( \int_0^{\infty} \bigg( \int_{\Btd} f \cdot \int_{\Btd} g \bigg)^q u(t)\,dt \bigg)^{1 / q}}{\|f\|_{p_1,v_1,\rn} \, \|g\|_{p_2,v_2,\rn}} & \\
    & \hspace{-5cm} = \sup_{t \in \I}{\big\| v_1^{- 1 / p_1} \big\|_{p_1^{\prime},\Btd}} \sup_{g \in \mp^+ (\rn)} \frac{\bigg( \int_0^{\infty} \bigg( \int_{\dual B(0,\tau)} g \bigg)^q \chi_{(0,t)}(\tau) u(\tau)\,d\tau \bigg)^{1 / q} }{\|g\|_{p_2,v_2,\rn}} \\
    & \hspace{-5cm} = \sup_{t \in \I} \big\| v_1^{- 1 / p_1} \big\|_{p_1^{\prime},\Btd} \bigg( \int_0^t u(y) \big\| v_2^{- 1} \big\|_{1,\Byd}^q \,dy \bigg)^{1 / q}.
    \end{align*}
	\end{proof}

\begin{thm}\label{main.thm}
	Let $1 \le p_1,\,p_2 < \infty$, $0 < q  < p_1$, $1/r_1 = 1/q - 1/p_1$. Suppose that $v_1,\,v_2 \in {\mathcal W}(\rn)$ are such that $\|v_i^{-{1} / {p_i}} \|_{p_i',\Btd} < \infty$, $t \in \I$ with $\lim_{t\rightarrow \infty} \|v_i^{-{1} / {p_i}} \|_{p_i',\Btd} = 0$, $i = 1,2$. Assume that $u\in {\mathcal W}(0,\infty)$ is such that $U^{{r_1} / {q}}$ is admissible and the fundamental function of the measure 
	$$
	d\nu (t) = U(t)^{r_1 / q}d \, \bigg[ - \big\| v_1^{- 1 / p_1} \big\|_{p_1^{\prime},\Btd}^{r_1 / p_1^{\prime}} \bigg]
	$$
	is non-degenerate with respect to $U^{{r_1} / {q}}$, that is, $\vp \in \Omega_{U^{r_1 / q}}$, where
	$$
	\vp (x) = \int_{[0,\infty)} \bigg( \frac{U(x)U(t)}{U(x) + U(t)}  \bigg)^{r_1 / q} d \, \bigg[ - \big\| v_1^{- 1 / p_1} \big\|_{p_1^{\prime},\Btd}^{r_1} \bigg], \quad x \in (0,\infty).
	$$
    Then inequality \eqref{eq.main} holds for all $f,\,g \in \mp^+ (\rn)$ if and only if:
	
	{\rm (i)} $p_2 \le q$, and
	\begin{align*}
	A_1 : = & \sup_{x \in (0,\infty)} \bigg(\int_0^{\infty}{\mathcal U}(x,t)^{{r_1}/{q}}\,  U(t)^{r_1 / q} \, d \, \bigg[ - \big\| v_1^{- 1 / p_1} \big\|_{p_1^{\prime},\Bxd}^{r_1} \bigg]\bigg)^{{1}/{r_1}} \\
	& \times \sup_{t\in (0,\infty)}{\mathcal U}(t,x)^{{1}/{q}} \big\| v_2^{-1/p_2}\big\|_{p_2^{\prime},(\Btd)} < \infty.
	\end{align*}
	Moreover, the best constant $C$ in \eqref{eq.main} satisfies $C\approx A_1$.
	
	{\rm (ii)} $q<p_2 \le r_1$, $1 / r_2 = 1 / q - 1 / p_2$, and
	\begin{align*}
	A_2 := & \sup_{x \in (0,\infty)} \bigg(\int_0^{\infty}{\mathcal U}(x,t)^{{r_1}/{q}}\,  U(t)^{r_1 / q} d \, \bigg[ - \big\| v_1^{- 1 / p_1} \big\|_{p_1^{\prime},\Btd}^{r_1} \bigg] \bigg)^{{1}/{r_1}} \\
	& \times \bigg(\int_0^{\infty}{\mathcal U}(t,x)^{{r_2} / q}
	\, d \bigg( -\big\| v^{-{1}/{p_2}} \big\|_{p_2^{\prime},(t-,\infty)}^{r_2} \bigg) \bigg)^{{1} / {r_2}} < \infty.
	\end{align*}
	Moreover, the best constant $C$ in \eqref{eq.main} satisfies $C\approx 	A_2$.
	
	{\rm (iii)} $r_1 < p_2 < \infty$, $1 / r_2 = 1 / q - 1 / p_2$, $1 / l = 1 / r_1 - 1 / p_2$, and
	\begin{align*}
	A_3 := & \bigg(\int_0^{\infty}\bigg(\int_0^{\infty}{\mathcal U}(x,t)^{{r_1}/{q}}\,  U(t)^{r_1 / q} d \, \bigg[ - \big\| v_1^{- 1 / p_1} \big\|_{p_1^{\prime},\Btd}^{r_1} \bigg] \bigg)^{(l-r_1) / {r_1}} \\
	& \times U(x)^{r_1 / q} \, \bigg( \int_0^{\infty} {\mathcal U}(t,x)^{{r_2} / q}
	\, d \bigg( -\big\| v^{-{1}/{p_2}} \big\|_{p_2^{\prime},(t-,\infty)}^{r_2} \bigg)
	\bigg)^{{l} / {r_2}} \, d \, \bigg[ - \big\| v_1^{- 1 / p_1} \big\|_{p_1^{\prime},\Bxd}^{r_1} \bigg] \bigg)^{{1} / {l}} < \infty.
	\end{align*}
	Moreover, the best constant $C$ in \eqref{eq.main} satisfies $C \approx A_3$.
	
	{\rm (iv)} $p_2 = \infty$, and
	\begin{align*}
	A_4 := & \left(\int_0^\infty\left( \int_0^\infty {\mathcal U(t,x)}\, d\bigg(-\big\|v_2^{-1}\big\|_{1,S[t,\infty)}^q\bigg)\right)^{r_1 / q}  U(x)^{r_1 / q} d \, \bigg[ - \big\| v_1^{- 1 / p_1} \big\|_{p_1^{\prime},\Bxd}^{r_1} \bigg] \right)^{{1} / {r_1}} < \infty.
	\end{align*}
	Moreover, the best constant $C$ in \eqref{eq.main} satisfies $C \approx A_4$.
\end{thm}

\begin{proof}
Assume that $\max\{1,q\} < p_1$ and $1 / r_1 = 1 / q - 1 / p_1$. By \cite[Theorem 2.2, (b) and (g)]{mu.emb}, \eqref{eq.iteration} yields that
\begin{align*}
\sup_{f,g \in \mp^+(\rn)} \frac{\bigg( \int_0^{\infty} \bigg( \int_{\Btd} f \cdot \int_{\Btd} g \bigg)^q u(t)\,dt \bigg)^{1 / q}}{\bigg(\int_{\rn} f^{p_1}  v_1 \bigg)^{1 / p_1} \, \bigg(\int_{\rn} g^{p_2}  v_2 \bigg)^{1 / p_2}} & \\
& \hspace{-5cm} = \, \sup_{g \in \mp^+ (\rn)} \frac{\bigg( \int_0^{\infty} \bigg( \int_0^x \bigg( \int_{\Btd} g \bigg)^q u(t) \,dt \bigg)^{r_1 / p_1} \, \bigg( \int_{\Bxd} g \bigg)^q u(x) \,  \big\| v_1^{- 1 / p_1} \big\|_{p_1^{\prime},\Bxd}^{r_1} \, dx \bigg)^{1 / r_1}}{\bigg(\int_{\rn} g^{p_2}  v_2 \bigg)^{1 / p_2}}.
\end{align*}
Integrating by parts, in view of $\lim_{t\rightarrow \infty} \|v_1^{-{1} / {p_1}} \|_{p_1',\Btd} = 0$, we have that
\begin{align*}
\bigg( \int_0^{\infty} \bigg( \int_0^x \bigg( \int_{\Btd} g \bigg)^q u(t) \,dt \bigg)^{r_1 / p_1} \, \bigg( \int_{\Bxd} g \bigg)^q u(x) \,  \big\| v_1^{- 1 / p_1} \big\|_{p_1^{\prime},\Bxd}^{r_1} \, dx \bigg)^{1 / r_1} & \\
& \hspace{-7cm} \approx \bigg( \int_0^{\infty}  \big\| v_1^{- 1 / p_1} \big\|_{p_1^{\prime},\Bxd}^{r_1} \, d \, \bigg( \int_0^x \bigg( \int_{\Btd} g \bigg)^q u(t) \,dt \bigg)^{r_1 / q} \bigg)^{1 / r_1} \\
& \hspace{-7cm} \approx \bigg( \int_0^{\infty}  \bigg( \int_0^x \bigg( \int_{\Btd} g \bigg)^q u(t) \,dt \bigg)^{r_1 / q} \, d \, \bigg[ - \big\| v_1^{- 1 / p_1} \big\|_{p_1^{\prime},\Bxd}^{r_1} \bigg] \bigg)^{1 / r_1}.
\end{align*}
Thus
\begin{align*}
\sup_{f,g \in \mp^+(\rn)} \frac{\bigg( \int_0^{\infty} \bigg( \int_{\Btd} f \cdot \int_{\Btd} g \bigg)^q u(t)\,dt \bigg)^{1 / q}}{\bigg(\int_{\rn} f^{p_1}  v_1 \bigg)^{1 / p_1} \, \bigg(\int_{\rn} g^{p_2}  v_2 \bigg)^{1 / p_2}} & \\
& \hspace{-5cm} = \,\sup_{g \in \mp^+(0,\infty)} \frac{\bigg( \int_0^{\infty}  \bigg( \int_0^x \bigg( \int_{\Btd} g \bigg)^q u(t) \,dt \bigg)^{r_1 / q} \, d \, \bigg[ - \big\| v_1^{- 1 / p_1} \big\|_{p_1^{\prime},\Bxd}^{r_1} \bigg] \bigg)^{1 / r_1}}{\bigg(\int_{\rn} g^{p_2}  v_2 \bigg)^{1 / p_2}} \\
& \hspace{-5cm} = \,\sup_{g \in \mp^+(0,\infty)} \frac{\bigg( \int_0^{\infty} \bigg( \frac{1}{U(x)}\int_0^x  \bigg( \int_{\Btd} g \bigg)^q u(t)\,dt \bigg)^{r_1 / q} \,  U(x)^{r_1 / q} d \, \bigg[ - \big\| v_1^{- 1 / p_1} \big\|_{p_1^{\prime},\Bxd}^{r_1} \bigg] \bigg)^{1/r_1}}{\bigg(\int_{\rn} g^{p_2}  v_2 \bigg)^{1 / p_2}}.  
\end{align*}

\begin{itemize}
\item[(i)] The statement follows by Theorem \ref{Thm.4.1}, (i). 

\item[(ii)] The statement follows by Theorem \ref{Thm.4.1}, (ii).

\item[(iii)] The statement follows by Theorem \ref{Thm.4.1}, (iv).

\item[(iv)] The statement follows by Theorem \ref{Thm.4.1}, (v).
\end{itemize}

\end{proof}

In the limiting case when $p_1 = \infty$ we obtain the following statement.
\begin{thm}\label{thm.main.00}
	Let $1 \le p_2 \le \infty$, $0 < q  < \infty$,  and let $u\in {\mathcal W}(0,\infty)$, $v_1,\,v_2 \in {\mathcal W}(\rn)$. Then inequality 
	\begin{equation}\label{eq.4.10.inf}
	\bigg( \int_0^{\infty} \bigg( \int_{\Btd} f \cdot \int_{\Btd} g \bigg)^q u(t)\,dt \bigg)^{1 / q} \leq C \, \|f\|_{\infty,v_1,\rn} \, \|g\|_{p_2,v_2,\rn}
	\end{equation}
	holds for all $f,\,g \in \mp^+ (\rn)$ if and only if:
	
	{\rm (i)} $p_2 \leq q$, and
	\begin{align*}
	D_1 := & \sup\limits_{t \in (0,\infty)} \bigg(\int_0^{\infty} u(y) \big\|v_1^{-1} \big\|_{1,\Byd}^q \, dy\bigg)^{{1} / {q}} \big\|v_2^{-{1} / {p_2}}\big\|_{p_2',\Btd} < \infty.
	\end{align*}
	Moreover, the best constant $C$ in \eqref{eq.4.10.inf} satisfies $C \approx D_1$.	
	
	{\rm (ii)} $q < p_2 < \infty$, $1 / r_2 = 1 / q - 1 / p_2$, and
	\begin{align*}
	D_2 := & \bigg( \int_0^{\infty} \bigg(\int_0^t u(\tau) \big\|v_1^{-1} \big\|_{1,{\dual B(0,\tau)}}^q \, d\tau \bigg)^{{r_2} / {q}} \, u(t) \, \big\|v_1^{-1} \, \big\|_{1,\Btd}^q \big\|v_2^{-{1} / {p_2}}\big\|_{p_2',\Btd}^{r_2} \,dt\bigg)^{1/r_2} < \infty.
	\end{align*}
	Moreover, the best constant $C$ in \eqref{eq.4.10.inf} satisfies $C \approx D_2$.
	
	{\rm (iii)} $p_2 = \infty$, and
	\begin{align*}
	D_3 := & \bigg( \int_0^{\infty}  u(t) \, \big\|v_1^{-1} \, \big\|_{1,\Btd}^q \big\|v_2^{-1}\big\|_{1,\Btd}^q \,dt\bigg)^{1 / q} < \infty.
	\end{align*}
	Moreover, the best constant $C$ in \eqref{eq.4.10.inf} satisfies $C \approx D_3$.
\end{thm}

\begin{proof}
	By \cite[Theorem 2.2, (e)]{mu.emb}, \eqref{eq.iteration} yields that
	\begin{align*}
	\sup_{f,g \in \mp^+(\rn)} \frac{\bigg( \int_0^{\infty} \bigg( \int_{\Btd} f \cdot \int_{\Btd} g \bigg)^q u(t)\,dt \bigg)^{1 / q}}{\|f\|_{\infty,v_1,\rn} \, \|g\|_{p_2,v_2,\rn}} = \, \sup_{g \in \mp^+ (\rn)} \frac{\bigg( \int_0^{\infty} \bigg( \int_{\Btd} g \bigg)^q u(t) \big\| v_1^{-1}\|_{1,\Btd}^q \,dt \bigg)^{1 / q} }{\|g\|_{p_2,v_2,\rn}}.
	\end{align*}
	The proof follows by application of \cite[Theorem 2.2]{mu.emb}.	
\end{proof}

We have the following statement when $q = \infty$.
\begin{thm}\label{thm.main.000}
	Let $1 \le p_1,\,p_2 \le \infty$,  and let $u\in {\mathcal W}(0,\infty)$, $v_1,\,v_2 \in {\mathcal W}(\rn)$. Then inequality 
	\begin{equation}\label{eq.4.10.sup}
	\esup_{t \in \I} \bigg( \int_{\Btd} f \cdot \int_{\Btd} g \bigg) u(t) \leq C \, \|f\|_{p_1,v_1,\rn} \, \|g\|_{p_2,v_2,\rn}
	\end{equation}
	holds for all $f,\,g \in \mp^+ (\rn)$ if and only if:
	
	{\rm (a)} $p_1,\,p_2 < \infty$, and
	\begin{align*}
	E_1 := & \sup_{t \in \I} u(t) \big\| v_1^{- 1 / p_1} \big\|_{p_1^{\prime},\Btd} \big\|v_2^{-{1} / {p_2}}\big\|_{p_2',\Btd} < \infty.
	\end{align*}
	Moreover, the best constant $C$ in \eqref{eq.4.10.sup} satisfies $C \approx E_1$.	

	{\rm (b)} $p_1 < \infty$, $p_2 = \infty$, and
    \begin{align*}
    E_2 := & \sup_{t \in \I} u(t) \big\| v_1^{- 1 / p_1} \big\|_{p_1^{\prime},\Btd} \big\|v_2^{-1}\big\|_{1,\Btd} < \infty.
    \end{align*}
    Moreover, the best constant $C$ in \eqref{eq.4.10.sup} satisfies $C \approx E_2$.	
	
	{\rm (c)} $p_1 = p_2 = \infty$, and
	\begin{align*}
	E_3 := &  \sup_{t \in \I} u(t) \big\| v_1^{- 1} \big\|_{1,\Btd} \big\|v_2^{-1}\big\|_{1,\Btd} < \infty.
	\end{align*}
	Moreover, the best constant $C$ in \eqref{eq.4.10.sup} satisfies $C \approx E_3$.
\end{thm}

\begin{proof}
(a) and (b): Let $p_1 < \infty$. By \cite[Theorem 2.2, (c) and (h)]{mu.emb}, \eqref{eq.iteration} yields that
\begin{align*}
\sup_{f,g \in \mp^+(\rn)} \frac{ \esup_{t \in \I} \bigg( \int_{\Btd} f \cdot \int_{\Btd} g \bigg) u(t) }{\|f\|_{p_1,v_1,\rn} \, \|g\|_{p_2,v_2,\rn}} & \\
& \hspace{-5cm} = \sup_{g \in \mp^+ (\rn)} \frac{1}{\|g\|_{p_2,v_2,\rn}} \sup_{t \in \I} \bigg( \esup_{\tau \in (0,t)} \bigg( \int_{\dual B(0,\tau)} g \bigg) u(\tau) \bigg) \,\big\| v_1^{- 1 / p_1} \big\|_{p_1^{\prime},\Btd}.
\end{align*}
Interchanging the suprema, by duality, on using \eqref{Fubini.1}, we arrive at
\begin{align*}
\sup_{f,g \in \mp^+(\rn)} \frac{ \esup_{t \in \I} \bigg( \int_{\Btd} f \cdot \int_{\Btd} g \bigg) u(t) }{\|f\|_{p_1,v_1,\rn} \, \|g\|_{p_2,v_2,\rn}} & \\
& \hspace{-5cm} = \sup_{t \in \I} \big\| v_1^{- 1 / p_1} \big\|_{p_1^{\prime},\Btd} \bigg( \esup_{\tau \in (0,t)}  u(\tau) \,\bigg( \sup_{g \in \mp^+ (\rn)} \frac{\int_{\dual B(0,\tau)} g}{\|g\|_{p_2,v_2,\rn}} \bigg) \bigg) \\
& \hspace{-5cm} = \sup_{t \in \I} \big\| v_1^{- 1 / p_1} \big\|_{p_1^{\prime},\Btd} \bigg( \esup_{\tau \in (0,t)}  u(\tau) \big\|v_2^{-{1} / {p_2}}\big\|_{p_2',\dual B(0,\tau)} \bigg) \\
& \hspace{-5cm} = \sup_{t \in \I} u(t) \big\| v_1^{- 1 / p_1} \big\|_{p_1^{\prime},\Btd} \big\|v_2^{-{1} / {p_2}}\big\|_{p_2',\Btd}, 
\end{align*}
when $p_2 < \infty$, and
\begin{align*}
\sup_{f,g \in \mp^+(\rn)} \frac{ \esup_{t \in \I} \bigg( \int_{\Btd} f \cdot \int_{\Btd} g \bigg) u(t) }{\|f\|_{p_1,v_1,\rn} \, \|g\|_{\infty,v_2,\rn}} & \\
& \hspace{-5cm} = \sup_{t \in \I} u(t) \big\| v_1^{- 1 / p_1} \big\|_{p_1^{\prime},\Btd} \big\|v_2^{-1}\big\|_{1,\Btd}, 
\end{align*}
when $p_2 = \infty$.

(c) Let $p_1 = p_2 = \infty$. By \cite[Theorem 2.2, (d)]{mu.emb}, \eqref{eq.iteration} yields that	
\begin{align*}
\sup_{f,g \in \mp^+(\rn)} \frac{ \esup_{t \in \I} \bigg( \int_{\Btd} f \cdot \int_{\Btd} g \bigg) u(t) }{\|f\|_{\infty,v_1,\rn} \, \|g\|_{\infty,v_2,\rn}} & \\
& \hspace{-5cm} = \sup_{g \in \mp^+ (\rn)} \frac{1}{\|g\|_{\infty,v_2,\rn}} \sup_{t \in \I} \bigg( \esup_{\tau \in (0,t)} \bigg( \int_{\dual B(0,\tau)} g \bigg) u(\tau) \bigg) \,\big\| v_1^{- 1} \big\|_{1,\Btd}.
\end{align*}
Interchanging the suprema, by duality, on using \eqref{Fubini.1}, we arrive at
\begin{align*}
\sup_{f,g \in \mp^+(\rn)} \frac{ \esup_{t \in \I} \bigg( \int_{\Btd} f \cdot \int_{\Btd} g \bigg) u(t) }{\|f\|_{\infty,v_1,\rn} \, \|g\|_{\infty,v_2,\rn}} & \\
& \hspace{-5cm} = \sup_{t \in \I} \big\| v_1^{- 1} \big\|_{1,\Btd} \bigg( \esup_{\tau \in (0,t)}  u(\tau) \,\bigg( \sup_{g \in \mp^+ (\rn)} \frac{\int_{\dual B(0,\tau)} g}{\|g\|_{\infty,v_2,\rn}} \bigg) \bigg) \\
& \hspace{-5cm} = \sup_{t \in \I} \big\| v_1^{- 1} \big\|_{1,\Btd} \bigg( \esup_{\tau \in (0,t)}  u(\tau) \big\|v_2^{-1}\big\|_{1,\dual B(0,\tau)} \bigg) \\
& \hspace{-5cm} = \sup_{t \in \I} u(t) \big\| v_1^{- 1} \big\|_{1,\Btd} \big\|v_2^{-1}\big\|_{1,\Btd}. 
\end{align*}

\end{proof}



\end{document}